\documentclass[a4paper,10pt]{article}
\usepackage[margin=1.5in]{geometry}

\usepackage{amsthm}
\newtheorem{theorem}{Theorem}

\usepackage{authblk}
\usepackage{amsmath}
\usepackage{cleveref}
\usepackage{lipsum}
\usepackage{amsfonts}
\usepackage{graphicx}
\usepackage{algorithm}
\usepackage{epstopdf}
\usepackage{algorithmic}
\usepackage{amssymb}
\usepackage{comment}
\usepackage{url}
\ifpdf
  \DeclareGraphicsExtensions{.eps,.pdf,.png,.jpg}
\else
  \DeclareGraphicsExtensions{.eps}
\fi
\usepackage{tikz}
\usetikzlibrary{matrix,fit,calc,positioning}
\usetikzlibrary{arrows,decorations.pathreplacing}
\usetikzlibrary{arrows.meta,patterns}
\usepackage[labelformat=simple]{subcaption}

\pgfkeys{tikz/mymatrixenv/.style={decoration={brace},every left delimiter/.style={xshift=8pt},every right delimiter/.style={xshift=-8pt}}}
\pgfkeys{tikz/mymatrix/.style={matrix of math nodes,nodes in empty cells,left delimiter={[},right delimiter={]},inner sep=1pt,outer sep=1.5pt,column sep=2pt,row sep=2pt,nodes={minimum width=15pt,minimum height=15pt,anchor=center,inner sep=0pt,outer sep=0pt}}}
\pgfkeys{tikz/mymatrixbrace/.style={decorate,thick}}

\newcommand*\mymatrixbraceright[4][m]{
 \draw[mymatrixbrace] (#1.east|-#1-#2-1.north east) -- node[right=2pt] {#4} (#1.east|-#1-#3-1.south east); }

\newcommand*\mymatrixbracetop[4][m]{
    \draw[mymatrixbrace] (#1.north-|#1-1-#2.north west) -- node[above=2pt] {#4} (#1.north-|#1-1-#3.north east);
}

\newcommand{\namealgnospace}{AT-MGRIT}
\newcommand{\namealg}{AT-MGRIT }
\usepackage{mathtools}

\newcommand{\timeint}{(t_0,t_f]}
\usepackage{array}
\usepackage{adjustbox}
\usepackage{multirow, makecell}

\title{Asynchronous Truncated Multigrid-reduction-in-time (AT-MGRIT) \thanks{This work is supported by the BMBF (project PASIROM; grant 05M18PXB). BSS was supported as a Nicholas C. Metropolis Fellow under the Laboratory Directed Research and Development program of Los Alamos National Laboratory.}}

\date{\vspace{-5ex}}
\author[1]{Jens Hahne}
\author[2]{Ben S. Southworth}
\author[1]{Stephanie Friedhoff}
\affil[1]{Department of Mathematics, Bergische Universität Wuppertal}
\affil[2]{Theoretical Division, Los Alamos National Laboratory, USA}
\affil[ ]{\textit {jens.hahne@math.uni-wuppertal.de, southworth@lanl.gov, friedhoff@math.uni-wuppertal.de}}

\usepackage{amsopn}

\begin{document}
\allowdisplaybreaks
\maketitle

\begin{abstract}
In this paper, we present the new ``asynchronous truncated multigrid-reduction-in-time'' (\namealgnospace) algorithm for introducing time
parallelism to the solution of discretized time-dependent problems. The new algorithm is based on the multigrid-reduction-in-time (MGRIT) approach, which, in certain settings, is equivalent to another common multilevel parallel-in-time method, Parareal. In contrast to Parareal and MGRIT that both consider a global temporal grid over the entire time interval on the coarsest level, the \namealg algorithm uses truncated local time grids on the coarsest level, each grid covering certain temporal subintervals. These local grids can be solved completely independent from each other, which reduces the sequential part of the algorithm and, thus, increases parallelism in the method. Here, we study the effect of using truncated local coarse grids on the convergence of the algorithm, both theoretically and numerically, and show, using challenging nonlinear problems, that the new algorithm consistently outperforms classical Parareal/MGRIT in terms of time to solution.
\end{abstract}

%
%

\section{Introduction}

Time-dependent problems are classically solved by a time-stepping procedure that propagates the solution stepwise forward in time. The method is optimal, i.\,e., of order $\mathcal{O}(N_t)$ for $N_t$ time steps. However, this method quickly becomes a parallel bottleneck when using modern computer architectures, which have an increasing number of processors, yet stagnating processor clock speed. Due to the sequential nature of classical time stepping, parallelization is limited to the spatial domain, and, as the number of processors grows, spatial parallelization becomes exhausted even if more resources are available. Parallel-in-time methods use these resources of modern computer architectures to compute multiple time steps simultaneously, enabling spatial \emph{and} temporal parallelization. 

The development of the first parallel-in-time method goes back over 50 years \cite{JNievergelt_1964}, and an overview of the field can be found in \cite{Gander2015_Review}. Two of the best known methods are the Parareal method \cite{JLLions_etal_2001a} and the multigrid-reduction-in-time (MGRIT) algorithm \cite{RDFalgout_etal_2014}, both of which are based on multigrid reduction principles \cite{Ries_MGR_methods} applied in the time dimension. Parareal can be interpreted as a two-level multigrid method, and MGRIT generalizes the approach to a multilevel setting. The ideas of both methods are similar, and both methods are equivalent in certain settings. On the ``fine'' level(s), time integration is simultaneously (i.\,e., in parallel) applied to non-overlapping temporal subdomains, and on the coarsest level, the entire time interval is solved with sequential time stepping. The choice of the number of levels and the choice of the coarsest grid is both critical and challenging. The typical choice of the coarse grid in the two-level setting is based on the number of processes, choosing as many points on the coarse grid as there are processes available \cite{JLLions_etal_2001a}. With this strategy, the fine level can be perfectly parallelized, but for a large number of processes, the serial work on the coarsest level dominates the runtime. 

Strategies to reduce the runtime of two-level schemes include variants of the Parareal algorithm, such as asynchronous Parareal  \cite{8253031,math6040045}, a modified version enhanced by the asynchronous iterative scheme \cite{CHAZAN1969199}, or an adaptive Parareal algorithm, which increases the accuracy of the fine solver over the Parareal iterations. Using more than two grid levels can significantly reduce the serial work by using a coarsest grid with only a few time points, but the resulting very large time steps can be very expensive, if not infeasible, to compute for some applications \cite{Bolten_etal_2020} and/or may affect the convergence of the algorithm \cite{danieli2021multigrid}.

MGRIT and Parareal are primarily effective on parabolic-type problems \cite{BOng_JBSchroder_2020,Southworth_etal_2020a}, which have a naturally dissipative behavior over long time intervals. Here, we make the observation that, due to the dissipative behavior inherent to these problems, the coarsest grid probably does \emph{not} need to represent the full time domain. Indeed, the solution at time $t=0$ will often have a negligible effect on the solution at much later times. Thus, in many cases we believe that computing a global coarse grid introduces an unnecessary sequential computational effort to an otherwise parallel algorithm.

In this paper, we introduce a new way to define the coarsest level in Parareal and MGRIT, emphasizing reducing the serial work while avoiding large time steps. Instead of solving the entire time interval serially on the coarsest grid, we define multiple \emph{independent} \emph{local} coarse grids each consisting of $k$ coarse-grid time points that can be propagated independently and simultaneously.\footnote{This work was originally motivated by similar processor-local multigrid hierarchies used in geometric and algebraic multigrid for elliptic problems \cite{bank2015algebraic,mitchell2018advances,mitchell2021parallel}.} Such an approach offers both improved parallelism and reduced computational cost compared with a global coarse-grid solve, while still providing sufficient coarse-level information to each processor for rapid convergence of the global problem. Due to the asynchronous nature of computing the truncated coarsest grids, we refer to the new algorithm as ``asynchronous truncated MGRIT'' (\namealgnospace).

\Cref{sec:main} introduces the algorithm in a two-level and multilevel context, providing an FAS interpretation of the multilevel variant in \Cref{alg:at_mgrit}. In \Cref{sec:theory}, we analyze the new algorithm theoretically, derive two-level error propagators, and present two-level convergence bounds in \Cref{sec:theory:conv}. We then describe various properties of the algorithm in \Cref{sec:numerical_experiments}, including describing the implementation with associated communication scheme in \Cref{sec:implementation} and performing a parameter study for a model problem in \ref{sec:heat_equation}. Finally, we apply the new algorithm to two challenging nonlinear problems, a chemical reaction in \Cref{sec:parallel_experiments:gray_scott} and the simulation of a realistic model of an electrical machine in \Cref{sec:parallel_experiments:electrical_machine}. \namealg consistently offers a 5--30\% reduction in wallclock time compared with traditional MGRIT and Parareal, and we expect the speedup to be greater if the algorithms were applied on GPUs.

\section{An overlapping and asynchronous coarse grid}
\label{sec:main}

Consider an initial value problem of the form
\begin{equation}\label{eq:problem_system}
    	\mathbf{u}'(t) = \mathbf{f}(t,\mathbf{u}(t)), \quad \mathbf{u}(t_0) = \mathbf{g}_0, \quad t \in (t_0,t_f].
\end{equation}
We discretize \eqref{eq:problem_system} on a uniformly-spaced temporal grid $t_i = i\Delta t, \; i=0,1,\dots,N_t$, with constant step size $\Delta t = (t_f-t_0)/N_t$, and let $\mathbf{u}_i \approx \mathbf{u}(t_i)$ for $i = 0,\dots,N_t$ with $\mathbf{u}_0 = \mathbf{u}(0)$. A general form of a single step time integration method for the time-discrete initial value problem is
\begin{equation}\label{eq:one_step_disc}
		\mathbf{u}_i = \boldsymbol{\Phi}_{i}(\mathbf{u}_{i-1}) + \mathbf{g}_i, \quad i = 1,2,\dots,N_t,
\end{equation}
where $\boldsymbol{\Phi}_{i}$ is a one-step time integrator, propagating a solution $\mathbf{u}_{i-1}$ from a time point $t_{i-1}$ to time point $t_{i}$, and $\mathbf{g}_i$ contains forcing terms. Equation \eqref{eq:one_step_disc} can be written as a semi-linear matrix equation
	\[
		A(\mathbf{u})\equiv\begin{bmatrix}
			I\\
			-\boldsymbol{\Phi}_{1}(\cdot) & I\\
			& \ddots & \ddots\\
			& & -\boldsymbol{\Phi}_{N_t}(\cdot) & I
		\end{bmatrix}\begin{bmatrix}
			\mathbf{u}_0\\
			\mathbf{u}_1\\
			\vdots\\
			\mathbf{u}_{N_t}
		\end{bmatrix} = \begin{bmatrix}
			\mathbf{g}_0\\
			\mathbf{g}_1\\
			\vdots\\
			\mathbf{g}_{N_t}
		\end{bmatrix}\equiv \mathbf{g},
	\]
where $\boldsymbol{\Phi}_{i}(\cdot)$ indicates that $\boldsymbol{\Phi}_{i}$ is nonlinearly evaluated at the corresponding (block) vector entry. This system can be solved by a (linear) sequential block forward solve.

In contrast, the iterative \namealg algorithm solves the problem by updating multiple time points simultaneously. In the following, we first introduce the idea of the algorithm in \Cref{sec:algorithm} for the two-level case and explain how the algorithm works. We then discuss how the two-level method can be extended to a multilevel setting.

\subsection{Two-level \namealg algorithm}\label{sec:algorithm}

For a given time grid $t_i = i\Delta t, i=0,1,...,N_{t}$, and a given coarsening factor $m > 1$, we define a splitting of all time-points into $F$- and $C$-points, such that every $m$-th point is a $C$-point (note, non-uniform coarsening is also possible; uniform coarsening is used here to simplify presentation). This defines a global coarse grid of $C$-points $T_i=i\Delta T, i=0,1,...,N_{T}$, with time step $\Delta T = m\Delta t$; all other non-$C$-points are $F$-points. Based on this global coarse grid, we define $N_T+1$ overlapping \emph{local} coarse grids. Given a distance $k$, the $p$th local coarse grid, $\mathcal{T}^{(p)}$ for $p=0,...,N_T$, is given by 
\begin{equation*}
\mathcal{T}^{(p)}=\{i\Delta T: i\in[\max(0,p-k+1),p]\},
\end{equation*}
with time step size $\Delta T = m\Delta t$, as depicted in \Cref{fig:time_grid}.
\begin{center}
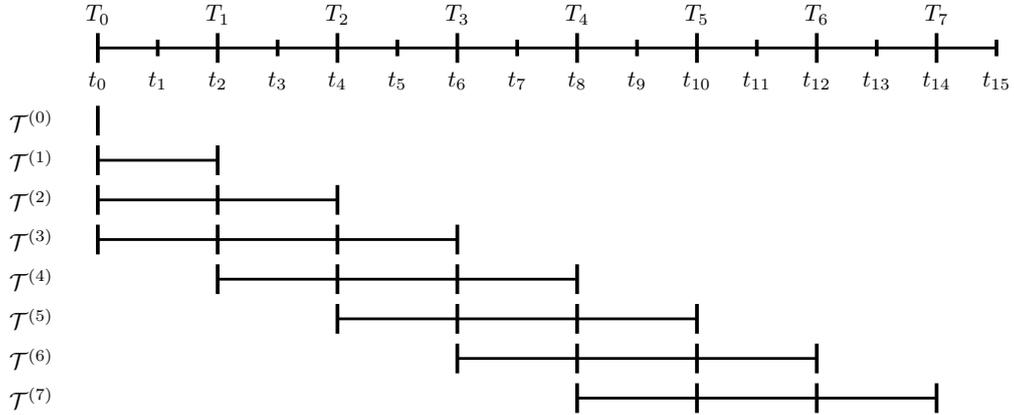
\begin{figure}
\centering
\begin{adjustbox}{width=\textwidth}
	\begin{tikzpicture}
	\def \j {0.9}
	\def \s {0}
	\def \m {2}
	\def \sl {-0.6}

	\draw[line width=1.2pt, -, >=latex'](0, \s) -- coordinate (x axis) (15*\j, \s) node[right] {}; 
	\foreach \x in {0,2,4,6,8,10,12,14} 
	{	
		\foreach \y in {1} 
		\draw[ultra thick, black] (\x*\j+\y*\j, \s+0.125) -- (\x*\j+\y*\j, \s-0.125) node[below] {} ;
		\draw[ultra thick, black] (\j*\x, \s+0.225) -- (\j*\x, \s-0.225) node[below] {} ;
	}

	\foreach \x in {0,1,...,15} 
	{	
		\draw (\j*\x,-.5) node {$t_{\x}$};
	}

		\foreach \x in {0,1,...,7} 
	{	
		\draw[black] (\j*2*\x,.5) node {$T_{\x}$};
	}
		
	\def \s {-0.5+1*\sl}
	\node at (-1,\s) {$\mathcal{T}^{(0)}$}; 
	\draw[line width=1.2pt, -, >=latex'](0,0+ \s) -- coordinate (x axis) (0*\j,0+ \s) node[right] {}; 
	\foreach \x in {0} 
		\draw[ultra thick, black] (\j*\m*\x,0.225+ \s) -- (\j*\m*\x,-0.225+ \s) node[below] {} ;

	\def \s {-0.5+2*\sl}
	\node at (-1,\s) {$\mathcal{T}^{(1)}$}; 
	\draw[line width=1.2pt, -, >=latex'](0,0+ \s) -- coordinate (x axis) (2*\j,0+ \s) node[right] {}; 
	\foreach \x in {0,1} 
		\draw[ultra thick, black] (\j*\m*\x,0.225+ \s) -- (\m*\j*\x,-0.225+ \s) node[below] {} ;

	\def \s {-0.5+3*\sl}
	\node at (-1,\s) {$\mathcal{T}^{(2)}$}; 
	\draw[line width=1.2pt, -, >=latex'](0,0+ \s) -- coordinate (x axis) (4*\j,0+ \s) node[right] {}; 
	\foreach \x in {0,1,2} 
		\draw[ultra thick, black] (\j*\m*\x,0.225+ \s) -- (\m*\j*\x,-0.225+ \s) node[below] {} ;

	\def \s {-0.5+4*\sl}
	\node at (-1,\s) {$\mathcal{T}^{(3)}$}; 
	\draw[line width=1.2pt, -, >=latex'](0,0+ \s) -- coordinate (x axis) (6*\j,0+ \s) node[right] {}; 
	\foreach \x in {0,1,2,3} 
		\draw[ultra thick, black] (\j*\m*\x,0.225+ \s) -- (\m*\j*\x,-0.225+ \s) node[below] {} ;

	\def \s {-0.5+5*\sl}
	\node at (-1,\s) {$\mathcal{T}^{(4)}$}; 
	\draw[line width=1.2pt, -, >=latex'](2*\j,0+ \s) -- coordinate (x axis) (8*\j,0+ \s) node[right] {}; 
	\foreach \x in {1,2,3,4} 
		\draw[ultra thick, black] (\j*\m*\x,0.225+ \s) -- (\m*\j*\x,-0.225+ \s) node[below] {} ;

	\def \s {-0.5+6*\sl}
	\node at (-1,\s) {$\mathcal{T}^{(5)}$}; 
	\draw[line width=1.2pt, -, >=latex'](4*\j,0+ \s) -- coordinate (x axis) (10*\j,0+ \s) node[right] {}; 
	\foreach \x in {2,3,4,5} 
		\draw[ultra thick, black] (\j*\m*\x,0.225+ \s) -- (\m*\j*\x,-0.225+ \s) node[below] {} ;

	\def \s {-0.5+7*\sl}
	\node at (-1,\s) {$\mathcal{T}^{(6)}$}; 
	\draw[line width=1.2pt, -, >=latex'](6*\j,0+ \s) -- coordinate (x axis) (12*\j,0+ \s) node[right] {}; 
	\foreach \x in {3,4,5,6} 
		\draw[ultra thick, black] (\j*\m*\x,0.225+ \s) -- (\m*\j*\x,-0.225+ \s) node[below] {} ;

	\def \s {-0.5+8*\sl}
	\node at (-1,\s) {$\mathcal{T}^{(7)}$}; 
	\draw[line width=1.2pt, -, >=latex'](8*\j,0+ \s) -- coordinate (x axis) (14*\j,0+ \s) node[right] {}; 
	\foreach \x in {4,5,6,7} 
		\draw[ultra thick, black] (\j*\m*\x,0.225+ \s) -- (\m*\j*\x,-0.225+ \s) node[below] {} ;		
		
	\end{tikzpicture}
	\end{adjustbox}
	\caption{\label{fig:time_grid}
Two-level temporal grid-hierarchy example for the \namealg algorithm with $N_t=15$, $m=2$ and $k=4$. The $C$-points (long markers) define the global coarse grid. For each point $p=0,\hdots,7$ on the global coarse grid, a local coarse grid $\mathcal{T}^{(p)}$ is created.}
\end{figure}\vspace{-3ex}
\end{center}
The two-level \namealg algorithm uses this time-grid hierarchy to solve time-dependent problems of the form \eqref{eq:one_step_disc} and is based on the following procedure: Given an initial solution $\mathbf{u}$ and the right-hand side $\mathbf{g}$, the first step of the algorithm applies a block relaxation, the so-called $F$-relaxation, to the fine space-time system of equations $A\textbf{u}= \textbf{g}$. The $F$-relaxation propagates the solution from a $C$-point to all following $F$-points preceding the next $C$-point (analogously to standard MGRIT/Parareal \cite{RDFalgout_etal_2014}). The relaxation of each interval of $F$-points can be executed in parallel and consists of $m-1$ sequential applications of the time integrator. In the next step, the global residual vector $\textbf{r}$ is computed and restricted by injection $(R_I^{(p)})$ to all local coarse grids. For each local coarse grid,the coarse system $A_c^{(p)}\textbf{u}^{(p)}_c= \textbf{r}^{(p)}_{c}$ is solved, which consists of $k-1$ sequential applications of the coarse time integrator. Since the coarse-grid problems are independent of each other, they can be solved simultaneously. Then, the global solution vector is corrected using ``selective ideal'' interpolation, $P_S^{(p)}$. The selective ideal interpolation is the transpose of an injection followed by an $F$-relaxation starting from exactly one point in time. More precisely, the approximation of the solution at the last time point of each local coarse grid is interpolated to the fine grid and then, an $F$-relaxation is performed using these interpolated points. The steps are applied iteratively until a desired quality of the solution is achieved. The two-level \namealg algorithm is summarized in \Cref{alg:two_level_at_mgrit}.
\begin{algorithm}
\caption{\label{alg:two_level_at_mgrit}\namealg($A, \mathbf{u}, \mathbf{g}$)}
\begin{algorithmic}[1]
\STATE \textbf{repeat}
\STATE \qquad Apply $F$-relaxation to $A\textbf{u}= \textbf{g}$
\STATE \qquad Compute residual $\textbf{r}=\textbf{g}-A\textbf{u}$
\STATE \qquad For $p=0$ to $N_T$:
\STATE \qquad \qquad Restrict residual, $\textbf{r}^{(p)}_c = R_I^{(p)}\textbf{r}$
\STATE \qquad \qquad Solve local system $A_c^{(p)}\textbf{u}^{(p)}_c= \textbf{r}^{(p)}_{c}$
\STATE \qquad \qquad Correct using $\textbf{u} = \textbf{u} + P_S^{(p)}\textbf{u}^{(p)}_{c} $
\STATE \textbf{until} stopping criterion is reached
\end{algorithmic}
\end{algorithm}

Note, that the \namealg algorithm solves for the exact solution in $N_T$ iterations if $k>1$. Furthermore, the algorithm is equivalent to the Parareal method if $k=N_T+1$, i.\,e., if all local coarse grids contain all $C$-points before in time. All components of the \namealg algorithm are highly parallel. The only communication needed is for the residual computation and the distribution of the residual (performed by the matrix-vector product $\textbf{r}^{(p)}_c = R_I^{(p)}\textbf{r}$ in \Cref{alg:two_level_at_mgrit}). Moreover, the coarse-level solve is communication-free (except for any communication that arises in spatial parallelism). This is particularly relevant for emerging heterogeneous computing architectures, where communication to and from GPU nodes can be quite expensive, and high efficiency is obtained with a low communication to computation ratio. For the coarse time integrator $\Phi_{i_c}$, here we choose a re-discretization of the problem with step size $\Delta T$, but other choices such as coarsening in space \cite{Ruprecht2014_GAMM,LunetEtAl2018,HowseEtAl2019} or order of discretization \cite{Nielsen2018,Falgout_etal_jcomp2019} can also be used.

\subsection{Multilevel FAS \namealg algorithm}\label{sec:algorithm2}

The two-level \namealg algorithm can easily be extended to the multilevel setting. 
Analogously to MGRIT, a multilevel hierarchy of temporal grids is constructed recursively using a uniform or non-uniform coarsening strategy. \namealg uses the same levels, coarsening, relaxation, and transfer operators as MGRIT on all finer levels in the hierarchy, but the coarsest MGRIT grid is replaced by local grids. \Cref{fig:multilevel_time_grid} shows an example grid hierarchy for three-level \namealg with $N_t=20$, $m=2$, and $k=4$. While MGRIT utilizes the global coarse grid on level 2, \namealg uses local grids $\mathcal{T}^{(2,p)}, p=0,\hdots,5$. 

In the following, we assume that all problem-dependent forcing terms are included in the time integrator. Then, the multilevel FAS \namealg $V$-cycle algorithm is given in \Cref{alg:at_mgrit}, where $N_t^{(\ell)}$ denotes the number of time points, and $\mathcal{A}^{(\ell)}\textbf{u}^{(\ell)}= \textbf{g}^{(\ell)}$ and $\mathcal{A}^{(\ell,p)}\textbf{u}^{(\ell,p)}=\textbf{g}^{(\ell,p)}$ specifies the space-time system of equations on levels $\ell=0,1,\hdots,L-1$ and on the local coarse grids $p=0,1,\hdots,N_t^{(\ell)}$, respectively. On all except for the coarsest level, we use restriction by injection $(R_{I}^{(\ell)})$, ``ideal'' interpolation $(P^{(\ell)})$, and $F(CF)^\nu$-relaxation. For more details on MGRIT (and thus \namealg on finer levels), see \cite{RDFalgout_etal_2014}. At the coarsest level, restriction and interpolation to and from the local coarse grids is done by injection, denoted by $R_I^{(\ell,p)}$ and $P_I^{(\ell,p)}$, respectively. Note that at the coarsest level, the residual is first transferred to the global coarse grid and then to the local coarse grids, allowing for a nicer notation of the algorithm. \namealg can also be used with other common MGRIT cycle types, such as F-cycles \cite{UTrottenberg_etal_2001a} or nested iterations \cite{GDahlquist_nes_it,LKronsjo_nes_it}. While $F$-cycles visit the coarsest level several times per iteration, nested iterations compute an improved initial guess by starting on the coarsest level and interpolating the solution to the finer levels, applying one $V$-cycle per level. For all cycle types, the standard MGRIT coarsest level can be replaced by local coarse grids.
Analogous to the two-level setting, \namealg is equivalent to MGRIT if $k=N_t^{(L-1)}+1$.
\begin{center}
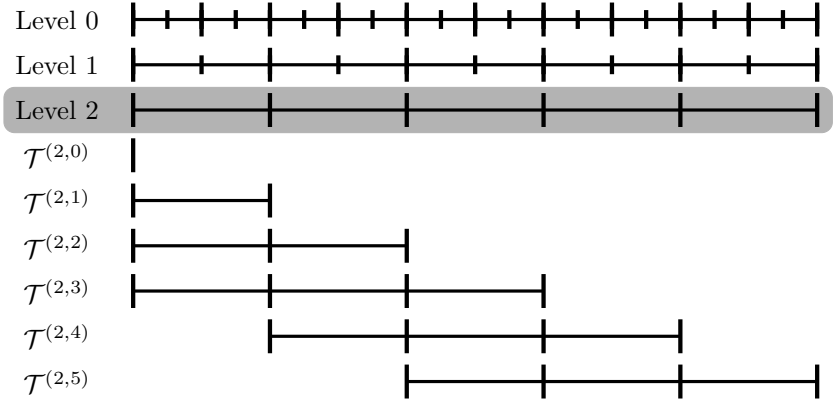
\begin{figure}
\centering

	\begin{tikzpicture}
	\def \j {0.45}
	\def \s {0}
	\def \m {2}
	\def \sl {-0.6}
	\node at (-1, \s) {Level 0}; 
	\draw[line width=1.2pt, -, >=latex'](0, \s) -- coordinate (x axis) (20*\j, \s) node[right] {}; 
	\foreach \x in {0,1,...,20} 
	{	
		\draw[ultra thick, black] (\x*\j, \s+0.125) -- (\x*\j, \s-0.125) node[below] {} ;
	}
	
	\foreach \x in {0,2,...,20} 
	{	
		\draw[ultra thick, black] (\j*\x, \s+0.225) -- (\j*\x, \s-0.225) node[below] {} ;
	}
	
	\def \s {1*\sl}
	\node at (-1, \s) {Level 1}; 
	\draw[line width=1.2pt, -, >=latex'](0, \s) -- coordinate (x axis) (20*\j, \s) node[right] {}; 
	\foreach \x in {0,2,...,20} 
	{	
		\draw[ultra thick, black] (\x*\j, \s+0.125) -- (\x*\j, \s-0.125) node[below] {} ;
	}
	
	\foreach \x in {0,4,...,20} 
	{	
		\draw[ultra thick, black] (\j*\x, \s+0.225) -- (\j*\x, \s-0.225) node[below] {} ;
	}

	\def \s {2*\sl}
	\draw[rounded corners,fill=gray!60,gray!60] (-1.7,\s-.3) rectangle (20*\j+.2,\s+.3);
	\node at (-1, \s) {Level 2}; 
	\draw[line width=1.2pt, -, >=latex'](0, \s) -- coordinate (x axis) (20*\j, \s) node[right] {};
	\foreach \x in {0,4,...,20} 
	{	
		\draw[ultra thick, black] (\j*\x, \s+0.225) -- (\j*\x, \s-0.225) node[below] {} ;
	}

	\def \s {3*\sl}
	\node at (-1,\s) {$\mathcal{T}^{(2,0)}$}; 
	\draw[line width=1.2pt, -, >=latex'](0,0+ \s) -- coordinate (x axis) (0*\j,0+ \s) node[right] {}; 
	\foreach \x in {0} 
		\draw[ultra thick, black] (\j*\m*\x,0.225+ \s) -- (\j*\m*\x,-0.225+ \s) node[below] {} ;

	\def \s {4*\sl}
	\node at (-1,\s) {$\mathcal{T}^{(2,1)}$}; 
	\draw[line width=1.2pt, -, >=latex'](0,0+ \s) -- coordinate (x axis) (4*\j,0+ \s) node[right] {}; 
	\foreach \x in {0,2} 
		\draw[ultra thick, black] (\j*\m*\x,0.225+ \s) -- (\m*\j*\x,-0.225+ \s) node[below] {} ;

	\def \s {5*\sl}
	\node at (-1,\s) {$\mathcal{T}^{(2,2)}$}; 
	\draw[line width=1.2pt, -, >=latex'](0*\j,0+ \s) -- coordinate (x axis) (8*\j,0+ \s) node[right] {}; 
	\foreach \x in {0,2,4} 
		\draw[ultra thick, black] (\j*\m*\x,0.225+ \s) -- (\m*\j*\x,-0.225+ \s) node[below] {} ;

	\def \s {6*\sl}
	\node at (-1,\s) {$\mathcal{T}^{(2,3)}$}; 
	\draw[line width=1.2pt, -, >=latex'](0*\j,0+ \s) -- coordinate (x axis) (12*\j,0+ \s) node[right] {}; 
	\foreach \x in {0,2,4,6} 
		\draw[ultra thick, black] (\j*\m*\x,0.225+ \s) -- (\m*\j*\x,-0.225+ \s) node[below] {} ;
		
	\def \s {7*\sl}
	\node at (-1,\s) {$\mathcal{T}^{(2,4)}$}; 
	\draw[line width=1.2pt, -, >=latex'](4*\j,0+ \s) -- coordinate (x axis) (16*\j,0+ \s) node[right] {}; 
	\foreach \x in {2,4,6,8} 
		\draw[ultra thick, black] (\j*\m*\x,0.225+ \s) -- (\m*\j*\x,-0.225+ \s) node[below] {} ;

	\def \s {8*\sl}
	\node at (-1,\s) {$\mathcal{T}^{(2,5)}$}; 
	\draw[line width=1.2pt, -, >=latex'](8*\j,0+ \s) -- coordinate (x axis) (20*\j,0+ \s) node[right] {}; 
	\foreach \x in {4,6,8,10} 
		\draw[ultra thick, black] (\j*\m*\x,0.225+ \s) -- (\m*\j*\x,-0.225+ \s) node[below] {} ;

	\end{tikzpicture}

	\caption{\label{fig:multilevel_time_grid}Example of a three-level time grid hierarchy for the \namealg algorithm for a fine grid with 21 time points, $m=2$ and $k=4$. At the coarsest level, a local coarse grid is generated for each $C$-point of the global coarse grid (gray box). These local grids $(\mathcal{T}^{(2,p)}, p = 0, \ldots, 5)$ replace the global coarse grid used in the classical MGRIT algorithm. 
}
\end{figure}\vspace{-3ex}
\end{center}

\begin{algorithm}
\caption{\label{alg:at_mgrit}\namealg FAS($\ell$)}
\begin{algorithmic}[1]
\STATE \textbf{repeat}
\STATE \qquad if $\ell$ is the coarsest level:
\STATE \qquad \qquad For $p=0$  to $N_t^{(\ell)}$:
\STATE \qquad \qquad \qquad Restrict to local grids
\STATE \qquad \qquad \qquad \qquad $\textbf{v}^{(\ell,p)} = R_I^{(\ell,p)}(\textbf{v}^{(\ell)})$
\STATE \qquad \qquad \qquad \qquad $\textbf{g}^{(\ell,p)} = R_I^{(\ell,p)}(\textbf{g}^{(\ell)})$
\STATE \qquad \qquad \qquad Solve local problem $\mathcal{A}^{(l,p)}(\textbf{u}^{(\ell,p)})= \mathcal{A}^{(\ell,p)}(\mathbf{v}^{(\ell,p)}) + \textbf{g}^{(\ell,p)}$
\STATE \qquad \qquad \qquad Update $\textbf{u}^{(\ell)} = P_I^{(\ell,p)}\textbf{u}^{(\ell,p)} $
\STATE \qquad else
\STATE \qquad \qquad Apply $F$-relaxation to $\mathcal{A}^{(\ell)}(\textbf{u}^{(\ell)})= \textbf{g}^{(\ell)}$
\STATE \qquad \qquad For $0$ to $\nu$:
\STATE \qquad \qquad \qquad Apply $CF$-relaxation to $\mathcal{A}^{(\ell)}(\textbf{u}^{(\ell)})= \textbf{g}^{(\ell)}$
\STATE \qquad \qquad Inject the approximation and its residual to the coarse grid
\STATE \qquad \qquad \qquad $\textbf{u}^{(l+1)}=R_{I}^{(\ell)}(\mathbf{u}^{(\ell)})$
\STATE \qquad \qquad \qquad $\textbf{v}^{(l+1)}=R_{I}^{(\ell)}(\mathbf{u}^{(\ell)})$
\STATE \qquad \qquad \qquad $\textbf{g}^{(l+1)}=R_{I}^{(\ell)}(\textbf{g}^{(\ell)}-\mathcal{A}^{(\ell)}\textbf{u}^{(\ell)})$
\STATE \qquad \qquad Solve on next level: \namealg($\ell+1$)
\STATE \qquad \qquad Compute the error approximation: $\mathbf{e}=\mathbf{u}^{(\ell+1)}-\mathbf{v}^{(\ell+1)}$
\STATE \qquad \qquad Correct using ideal interpolation: $\textbf{u}^{(\ell)} = \textbf{u}^{(\ell)} + P^{(\ell)}(\textbf{e})$ 
\STATE \textbf{until} stopping criterion is reached
\end{algorithmic}
\end{algorithm}

\section{Theory}\label{sec:theory}

This section develops convergence theory for \namealg in the linear two-level setting. The analysis is built on two-level MGRIT/Parareal theory developed in \cite{doi:10.1137/16M1074096,doi:10.1137/18M1226208}, and gives insight on the effects of truncating the coarse-grid time grid. We begin by introducing the error-propagation operator in the case of exact solves on a truncated coarse grid (\Cref{sec:theory:err:exact}) and inexact coarse-grid solves (\Cref{sec:theory:err:approx}). Formal two-level convergence bounds are provided in \Cref{sec:theory:conv}. Because we are in the two-level setting, we drop $\ell$ superscripts from \Cref{sec:algorithm2}.

\subsection{Error propagation}\label{sec:theory:err}

Following from \cite{RDFalgout_etal_2014}, the two-level error propagation operator for linear \namealg with an exact coarse-grid solve is given by:
\begin{align}\label{eq:E0}
\mathcal{E} & := \left(I - \sum_{p=0}^{N_T} P_S^{(p)}( A_c^{(p)})^{-1} R_I^{(p)} A \right)PR_I, 
\end{align} 
where $A_c^{(p)}$ represents the local coarse grid systems, $R_I^{(p)}$ is the restriction operator to the local coarse grids, and $P_S^{(p)}$ defines the interpolation from the local coarse grids that updates the fine grid using \emph{selective ideal interpolation}, i.\,e., for one specific {$C$-point, this $C$-point and the following interval of $F$-points are updated.} We see that \eqref{eq:E0} is analogous to that derived in \cite[Eq. 2.12]{RDFalgout_etal_2014}, but here we must sum over {$C$-points, as each $C$-point is updated by a unique local coarse-grid.}
The operators $P$ and $R_I$, corresponding to ideal interpolation and restriction by injection, respectively, are given by

\begin{center}
\begin{tikzpicture}[baseline={-0.5ex},mymatrixenv]

\matrix [mymatrix,inner sep=4pt,label={[]left:$P\coloneqq I \otimes$},label={[]right:$\;\;\;\; ,$}] (m) at  (-2.01,0)
{
	I \\
	\Phi   \\
	\vdots \\
	\Phi^{m-1} \\ 
};

\mymatrixbraceright{1}{4}{$\scriptscriptstyle m$}
        
\matrix [mymatrix,inner sep=4pt,label={[]left:$R_I\coloneqq$},label={[]right:$\;\;\;\; \; \;\; ,$}] (m)at  (-0,-3.2)  
{
	I & &&&&&&&& \\
	& \mathbf{0} & ... & \mathbf{0} & I & &&&& \\
	&   & &&&\ddots&&&& \\ 
	&&&&  &  & \mathbf{0} & ... & \mathbf{0} & I  \\
};

\mymatrixbracetop{2}{5}{$\scriptscriptstyle m$}
\mymatrixbraceright{1}{4}{$\scriptscriptstyle N_t^{(1)}$}

\matrix [mymatrix,inner sep=4pt,label={[]left:$P_S^{(p)}\coloneqq$},label={[]right:$\;\;\;\; .$}] (m) at  (6,-1.3)
{
	&&  \\
	&& \\ 
	&& I \\
	&& \Phi \\
	&& \vdots \\
	&& \Phi^{m-1} \\
	&&  \\
	&&  \\ 
};

\mymatrixbraceright{1}{2}{$\scriptscriptstyle pm$}
\mymatrixbraceright{3}{6}{$\scriptscriptstyle m$}
\mymatrixbracetop{1}{2}{$\scriptscriptstyle \min(p,k-1)$}

\end{tikzpicture}
\end{center}
Note, that the operator $PR_I$ is equivalent to error propagation for $F$-relaxation \cite{RDFalgout_etal_2014}. 

\subsubsection{Exact local coarse grid solve}\label{sec:theory:err:exact}

First, we consider the effect of the local coarse grids using exact solves on the coarse time steps. For this purpose, we define the local coarse-grid problem as
\begin{align*}
A_c^{(p)} &\coloneqq  R_I^{(p)} AP^{(p)},
\end{align*}
where $P^{(p)}$ and $R_I^{(p)}$ define the transfer between the fine grid and the local coarse grids and are submatrices of $P$ and $R_I$. For $P^{(p)}$, only columns of $P$ associated to points lying on this local coarse grid are considered. Equivalently, only the associated rows are considered for the restriction. Then, the coarse-grid problems are given by

\begin{center}\label{eq:RAP}
\begin{tikzpicture}[baseline={-0.5ex},mymatrixenv]

\matrix [mymatrix,inner sep=4pt,label={[]left:$R_I^{(p)}AP^{(p)}=$}] (m) at (6.3,0)
{
	I & & & & \\ -
	\Phi^{m} & I \\ 
	& -\Phi^{m} & I \\
	&& \ddots & \ddots \\ 
	&&& -\Phi^{m} & I \\
};

\mymatrixbraceright{1}{5}{$\scriptscriptstyle \min(p+1,k)$}
\mymatrixbracetop{1}{5}{$\scriptscriptstyle \min(p+1,k)$}
\end{tikzpicture}.
\end{center}
Here, it is important to note that all local coarse-grid systems $R_I^{(p)}AP^{(p)}$ have the same structure, but consider different time intervals. In fact, the exact local coarse-grid systems are principle submatrices of the Schur complement corresponding to a standard Parareal/MGRIT coarse-grid with exact solves \cite{RDFalgout_etal_2014}.

We can now examine the error-propagator $\mathcal{E}_{e}$ using the exact local coarse grids. We refer to \Cref{app:AppendixA} for detailed algebraic derivations. In forming $\mathcal{E}_{e}$ by summing over all $p=0,1,\hdots,N_T$, we obtain a block lower triangular matrix, whereby each $p$ updates $m$ rows of $\mathcal{E}$, and the error-propagator using ideal local coarse grids can be written in block form as
\begin{center}
\begin{equation}\label{eq:error_exact}
\begin{tikzpicture}[baseline={-0.5ex},mymatrixenv]

\matrix [mymatrix,inner sep=4pt,label={[]left:$\mathcal{E}_e =$},label={[]right:$\;\;\;\; ,$}] (m)
{
	\mathbf{0}  &\hdots     &            &            &        &\mathbf{0} \\ 
	\vdots      &\hdots     &            &            &        &\vdots \\
	\mathbf{0}  &\hdots     &            &            &        &\mathbf{0}  \\
	\mathcal{G} &\mathbf{0} &\hdots      &            &        &\mathbf{0} \\ 	
	\mathbf{0}  &\ddots     &\ddots	     &            &        &\vdots \\
	\mathbf{0}  &\mathbf{0} &\mathcal{G} & \mathbf{0} & \hdots & \mathbf{0} \\
};

\mymatrixbraceright{1}{3}{$\scriptscriptstyle km$}
\mymatrixbraceright{4}{4}{$\scriptscriptstyle m$}
\mymatrixbracetop{4}{6}{$\scriptscriptstyle km$}
\mymatrixbracetop{1}{1}{$\scriptscriptstyle m$}

\matrix [mymatrix,inner sep=4pt,label={[]left:$\mathcal{G}=$},label={[]right:$\;\;\;\; .$}] (m)at  (6,0)  
{
\Phi^{km} & \mathbf{0} & \hdots & \mathbf{0} \\
\Phi^{km+1} & \mathbf{0} & \hdots & \mathbf{0} \\
\vdots & \vdots & \vdots & \vdots \\
\Phi^{km+m-1} & \mathbf{0} & \hdots & \mathbf{0} \\
};

\mymatrixbracetop{1}{4}{$\scriptscriptstyle m$}
\mymatrixbraceright{1}{4}{$\scriptscriptstyle m$}

\end{tikzpicture}
\end{equation}
\end{center}
Note that the error propagator $\mathcal{E}_e$ is nonzero, so unlike Parareal/two-level MGRIT, \namealg using exact local coarse-grid inverses is not a direct method. For all $p>k$, we have some error perturbation that results from truncating the exact (Schur-complement) coarse grid.

\subsubsection{Approximate local coarse grid solve}\label{sec:theory:err:approx}

As typical for multigrid reduction techniques, we do not invert $R_I^{(p)}AP^{(p)}$ exactly, but approximate $R_I^{(p)} AP^{(p)}\approx\widetilde{A}_c^{(p)}$. Specifically, we approximate the powers $\Phi^{m}$, which correspond to $m$ applications of the fine time integrator $\Phi$, with a coarse operator $\Psi$. This results in the approximation $\widetilde{A}_c^{(p)}$ given by
\begin{center}
\begin{tikzpicture}[baseline={-0.5ex},mymatrixenv]

\matrix [mymatrix,inner sep=4pt,label={[]left:$\widetilde{A_c}^{(p)} \coloneqq$}] (m)
{
	I&&& \\
 	-\Psi & I&& \\
	& \ddots & \ddots \\
	&& -\Psi & I \\
};

\mymatrixbraceright{1}{4}{$\scriptscriptstyle \min(p+1,k)$}
\mymatrixbracetop{1}{4}{$\scriptscriptstyle \min(p+1,k)$}

\end{tikzpicture}.
\end{center}
Using this approximation, we can formulate the error-propagation $\mathcal{E}_a$ using the approximated local coarse-grid inverse. Again, we refer to \Cref{app:AppendixB} for derivations. The error-propagation operator with approximate coarse grid, $\mathcal{E}_a$, is then given by
\begin{center}
\begin{equation}\label{eq:error_a}
\begin{tikzpicture}[baseline={-0.5ex},mymatrixenv]

\matrix [mymatrix,inner sep=4pt,label={[]left:$\mathcal{E}_a =$},label={[]right:$\; \; \; \; ,$}] (m)
{
  \mathbf{0} & \hdots & &  & & &  \mathbf{0} \\
  \mathcal{Z}_0 & \mathbf{0} & \hdots & & &  &  \mathbf{0} \\
\vdots & \ddots & \ddots & \hdots & &  &  \mathbf{0} \\
\mathcal{Z}_{k-2} & \hdots & \mathcal{Z}_0 & \mathbf{0} & \hdots & &   \mathbf{0} \\
\mathcal{W} & \mathcal{Z}_{k-2} & \hdots & \mathcal{Z}_0 & \mathbf{0} & \hdots &   \mathbf{0} \\
\mathbf{0} & \ddots & \ddots & \vdots & \ddots  & \ddots &  \vdots \\
\mathbf{0} &  \mathbf{0}  & \mathcal{W} & \mathcal{Z}_{k-2} & \hdots & \mathcal{Z}_0 &  \mathbf{0} \\
};

\mymatrixbraceright{1}{4}{$\scriptscriptstyle km$}
\mymatrixbraceright{5}{7}{$\scriptscriptstyle (P-k)m$}
\mymatrixbracetop{1}{3}{$\scriptscriptstyle (P-k)m$}
\mymatrixbracetop{4}{7}{$\scriptscriptstyle km$}

\end{tikzpicture}
\end{equation}
\end{center}
with block matrices $\mathcal{Z}_x$ and $\mathcal{W}$ given by
\begin{center}
\begin{equation}\label{eq:error_inexact_z_w}
\begin{tikzpicture}[baseline, mymatrixenv]

\matrix [mymatrix,inner sep=4pt,label={[]left:$\mathcal{Z}_x =$},label={[]right:,}] (m)
{
	\Phi^{0}\Psi^{x}(\Phi^{m}-\Psi) & \mathbf{0} & \hdots & \mathbf{0} \\
	\vdots & \vdots & \vdots & \vdots\\
	\Phi^{m-1}\Psi^{x}(\Phi^{m}-\Psi) & \mathbf{0} & \hdots & \mathbf{0} \\
};

\mymatrixbracetop{1}{4}{$\scriptscriptstyle m$}

\matrix [mymatrix,inner sep=4pt,label={[]left:$\mathcal{W} =$},label={[]right:.}] (m)at  (5.8,0)  
{
	\Phi^{0}\Psi^{k-1}\Phi^{m} & \mathbf{0} & \hdots & \mathbf{0} \\
	\vdots & \vdots & \vdots & \vdots\\
	\Phi^{m-1}\Psi^{k-1}\Phi^{m} & \mathbf{0} & \hdots & \mathbf{0} \\
};

\mymatrixbracetop{1}{4}{$\scriptscriptstyle m$}

\end{tikzpicture}
\end{equation}
\end{center}

\subsection{Convergence bounds}\label{sec:theory:conv}

To avoid multiple subscripts, let $\mathcal{E} \mapsto \mathcal{E}_a$
from \eqref{eq:error_a}. Using $f$ and $c$ subscripts to denote $F$- and $C$-points, respectively, $\mathcal{E}$ \eqref{eq:error_a} can be partitioned into $2\times 2$ block form. If we additionally consider powers of the matrix, which correspond to several iterations, we get
\begin{align*}
\mathcal{E}^\ell \coloneqq \begin{bmatrix} \mathcal{E}_{ff} & \mathcal{E}_{fc} \\
	\mathcal{E}_{cf} & \mathcal{E}_{cc}\end{bmatrix}^\ell 
= \begin{bmatrix} \mathbf{0}& \mathcal{E}_{fc} \\
	\mathbf{0} & \mathcal{E}_{cc}\end{bmatrix}^\ell
= \begin{bmatrix} \mathbf{0}& \mathcal{E}_{fc}\mathcal{E}_{cc}^{\ell-1} \\
	\mathbf{0} & \mathcal{E}_{cc}^\ell \end{bmatrix}.
\end{align*}
It follows from above that for multiple iterations, convergence is fully
defined by $\mathcal{E}_{cc}$, that is, $\mathcal{E}^\ell$ will be convergent
in some norm or measure if and only if $\mathcal{E}_{cc}^\ell$ is as well. To
that end, we consider analyzing the C-C principle submatrix of
\eqref{eq:error_a},
{\small
\begin{align}\label{eq:Ecc}
&\hspace{-85ex}\mathcal{E}_{cc} = \\\nonumber
\begin{bmatrix} 
  \mathbf{0} & & &  & & &  \\
  (\Phi^{m}-\Psi) & \mathbf{0} & & & & &  \\
\vdots & (\Phi^{m}-\Psi) & \ddots & & & &  \\
\Psi^{k-2}(\Phi^{m}-\Psi) & \vdots & \ddots & \mathbf{0} & \\
\Psi^{k-1}\Phi^{m} & \Psi^{k-2}(\Phi^{m}-\Psi) & \hdots & 
	(\Phi^{m}-\Psi) & \mathbf{0} & &  \\
& \ddots & \ddots  & \vdots & \ddots & \ddots \\
 	&  &\Psi^{k-1}\Phi^{m} & \Psi^{k-2}(\Phi^{m}-\Psi) &
	\hdots & (\Phi^{m}-\Psi) &  \mathbf{0}
\end{bmatrix}.&
\end{align}
}
Now consider the case of $\Phi$ and $\Psi$ being simultaneously diagonalizable, as
would occur if the same (diagonalizable) spatial matrix is used on the fine and
coarse grid. Let $U$ denote the shared eigenvector matrix of $\Phi$ and $\Psi$,
with eigenvalues $\mu\in\sigma(\Psi)$ and $\lambda\in\sigma(\Phi)$, where 
$\sigma(\Psi)$ and $\sigma(\Phi)$ denote the spectrum of $\Psi$ and $\Phi$,
respectively. Following the frameworks developed in \cite{doi:10.1137/16M1074096,
doi:10.1137/18M1226208}, let $\widetilde{U}$ denote a block-diagonal matrix,
with diagonal blocks given by eigenvectors $U$. Then, 
\begin{equation*}
\|\mathcal{E}_{cc}\|_{(\widetilde{U}\widetilde{U}^*)^{-1}} =
\max_{\{\mu,\lambda\}} \|\widetilde{\mathcal{E}}_{cc}\|,
\end{equation*}
where $\widetilde{\mathcal{E}}_{cc}$ is defined as follows for a fixed
pair of eigenvalues $\{\mu,\lambda\}$:
\begin{align*}
&\widetilde{\mathcal{E}}_{cc} \coloneqq\\
&\begin{bmatrix} 
  \mathbf{0} & & &  & & &  \\
  (\lambda^{m}-\mu) & \mathbf{0} & & & & &  \\
\vdots & (\lambda^{m}-\mu) & \ddots & & & &  \\
\mu^{k-2}(\lambda^{m}-\mu) & \vdots & \ddots & \mathbf{0} & \\
\mu^{k-1}\lambda^{m} & \mu^{k-2}(\lambda^{m}-\mu) & \hdots & 
	(\lambda^{m}-\mu) & \mathbf{0} & &  \\
& \ddots & \ddots  & \vdots & \ddots & \ddots \\
 	&  &\mu^{k-1}\lambda^{m} & \mu^{k-2}(\lambda^{m}-\mu) &
	\hdots & (\lambda^{m}-\mu) &  \mathbf{0} \\
\end{bmatrix}.\nonumber
\end{align*}

If the spatial matrix is normal, then $(\widetilde{U}\widetilde{U}^*)^{-1} = I$.
In general, bounding $\widetilde{\mathcal{E}}_{cc}$ in the $\ell^2$-norm 
for each eigenvalue pair guarantees convergence of $\mathcal{E}_{cc}$
in a certain eigenvector induced norm.

Recall the inequality $\|\widetilde{\mathcal{E}}_{cc}\|^2 \leq
	\|\widetilde{\mathcal{E}}_{cc}\|_1\|\widetilde{\mathcal{E}}_{cc}\|_\infty$. 
Given that $\widetilde{\mathcal{E}}_{cc}$ is Toeplitz, the maximum row and
column sums are equal, yielding the bound

\begin{align*}
\|\widetilde{\mathcal{E}}_{cc}\| \leq 
	\|\widetilde{\mathcal{E}}_{cc}\|_1 
& = |\lambda^{m} - \mu | \sum_{\ell=0}^{k-2} |\mu^\ell| + |\lambda^{m} \mu^{k-1}| \\
& \hspace{-5ex} = \frac{|\lambda^{m}-\mu |(1 - |\mu|^{k-1})}{1 - |\mu|} + |\lambda^{m} \mu^{k-1}|.
\end{align*}
Results are summarized in the following theorem.

\begin{theorem}[Two-level convergence]
Let $\Phi$ and $\Psi$ be simultaneously diagonalizable with eigenvectors $U$,
and consider two-level \namealg with coarsening factor $m$ and local
coarse-grid size $k$. For a given CF-splitting of time points, error
propagation takes the form

\begin{align*}
\mathcal{E}^\ell  \coloneqq \begin{bmatrix} \mathcal{E}_{ff} & \mathcal{E}_{fc} \\
	\mathcal{E}_{cf} & \mathcal{E}_{cc}\end{bmatrix}^\ell 
= \begin{bmatrix} \mathbf{0}& \mathcal{E}_{fc}\mathcal{E}_{cc}^{\ell-1} \\
	\mathbf{0} & \mathcal{E}_{cc}^\ell \end{bmatrix}.
\end{align*}
Further, assume $\Phi$ and $\Psi$ are stable in an eigenvalue
sense, that is, $|\mu|,|\lambda|<1$ for all $\mu\in\sigma(\Psi)$,
$\lambda\in\sigma(\Phi)$. Then, for eigenvalue pairs (with shared
eigenvector) $\{\mu,\lambda\}$,
\begin{align}\label{eq:Ecc_bound}
\|\mathcal{E}_{cc}\|_{(\widetilde{U}\widetilde{U}^*)^{-1}} \leq
\max_{\{\mu,\lambda\}} \left( \frac{|\lambda^{m}-\mu |(1 - |\mu|^{k-1})}
	{1 - |\mu|} + |\lambda^{m} \mu^{k-1}| \right).
\end{align}

\end{theorem}
\begin{proof}
The proof follows from the above discussion.
\end{proof}

Note that the first term of equation \eqref{eq:Ecc_bound} resembles (necessary and sufficient \cite{doi:10.1137/18M1226208}) convergence bounds for two-level Parareal and MGRIT \cite{doi:10.1137/16M1074096,doi:10.1137/18M1226208}, while the second term introduces an error perturbation that results from truncating the coarse grid. We now make three observations suggesting the additional error perturbation over standard Parareal/MGRIT will not significantly degrade convergence, which are verified numerically in \Cref{sec:numerical_experiments,sec:parallel_experiments}.
\begin{enumerate}
	\item For $|\mu|\ll 1$ and $|\lambda| < 1$, for quite reasonable $m$ and $k$, we will have the error term $|\lambda^{m} \mu^{k-1}|\ll 1$ (even if, say, $|\lambda|$ is close to one).

	\item For $|\mu|$ close to one, recall that two-level convergence of MGRIT/Parareal requires
	$|\mu - \lambda^{m}| \leq \varphi (1 - |\mu|)$, with constant $\varphi$ \cite{doi:10.1137/18M1226208}. Thus, a
	successful Parareal/MGRIT scheme requires $|\mu - \lambda^{m}| \approx 0$, in which
	case $|\lambda^{m} \mu^{k-1}| \approx |\lambda|^{km}$. Because $|\lambda| < 1$ by
	assumption, the truncated coarse-grid error term still rapidly becomes small for moderate
	$k$ and $m$. For example let $|\lambda| = |\mu| = 0.99$, $k = 0.1P$ for $P$
	processes (see \Cref{sec:heat_equation}), and $m = (N_t+1)/P$. Then, for 5,000
	time steps, we have $|\mu^{k-1}\lambda^{m}|\lessapprox 0.99^{500} = 0.0065...$  

	\item Finally, regardless of the exact size of the additional error term, it does not have significant impact on long-term convergence behavior. Multiplication of Toeplitz matrices such as $\widetilde{\mathcal{E}}_{cc}^\ell$ corresponds to finite discrete convolutions. With some algebra, one can show that the ``error'' subdiagonal, that is, the subdiagonal of $\widetilde{\mathcal{E}}_{cc}$ that lacks a $\mu-\lambda^{m}$ scaling, is propagated out of the matrix after $P/k$ iterations 
	(i.e., all matrix entries then have at least one power of $\mu-\lambda^{m}$).
	As a result, if we really do have divergent eigenmodes $\{\mu,\lambda\}$ for which
	the bound in equation \eqref{eq:Ecc_bound} $> 1$, then after a small number of iterations
	(e.g., 5-10 for $k/P\sim 0.1-0.2$) such non-convergent modes will begin converging,
	indeed \emph{before} the initial condition has propagated across all processes.

\end{enumerate}

\section{Algorithmic properties}\label{sec:numerical_experiments}

This section examines nuances of the \namealg algorithm, including a communication scheme for distributing residuals on the coarsest level (\Cref{sec:implementation}), the implicit propagation of initial conditions across the time domain (\Cref{sec:propagation}), and a study on the new parameter $k$ (\Cref{sec:heat_equation}).

\subsection{Implementation}\label{sec:implementation}

We have implemented \Cref{alg:at_mgrit} in parallel as part of the \texttt{PyMGRIT} package \cite{PyMGRIT,Bolten_etal_2020} framework. When applying the algorithm in parallel, we assume that at the coarsest level at most one $C$-point of the global grid lies on one process. In principle this is not necessary, but ensures that the solves of the local problems can be perfectly parallelized. To solve the local problems, the fine(r)-level residuals must be distributed. Let $p_0,p_1,...,p_{N_T}$, where $N_T$ is equal to the number of points on the coarsest global coarse grid, be all processes containing a $C$-point on the coarsest grid. Then, we define two groups of MPI communicators. The first decomposition divides all processes based on the distance $k$ into $N_T/k$ groups, where the first group consists of processes $p_0,p_1,...,p_{k-1}$, the second consists of processes $p_k,p_{k+1},...,p_{2k-1}$, and so on. The second decomposition divides the processes $p_{k-1},...,p_{N_T}$ into groups of size $k$, so that the first group consists of processes $p_{k-1},...,p_{2k-2}$, the next of processes $p_{2k-1},...,p_{3k-1}$, and so on. Then, the distribution of residuals is given by a communication within all groups of the first decomposition, followed by a communication within all groups of the second decomposition. Note that the groups of a decomposition do not overlap, allowing for parallel communication within each group. \Cref{fig:communication} shows an example of the residual communication for a two-level \namealg algorithm with a global coarse grid with $N_T=6$ time points, along with a description of the communication stages.

\begin{figure}[bht!]
\centering
\begin{adjustbox}{width=\textwidth}
  \begin{tikzpicture}[y=1cm, x=1cm, thick, font=\footnotesize]    

\node (b) at (-1,0) {Fine};
\def\w{0.2}
\draw[line width=1.2pt, -, >=latex'](0,0) -- coordinate (x axis) (\w*17,0) node[right] {}; 
\foreach \x in {0,1,2,...,17} \draw[ultra thick, black] (\w*\x,0.125) -- (\w*\x,-0.125) node[below] {} ;
\foreach \x in {0,3,6,9,12,15} \draw[ultra thick, black] (\w*\x,0.175) -- (\w*\x,-0.175) node[above] {} ;

\def\y{-0.5}
\node (b) at (-1,\y) {Proc 0};
\draw[line width=1.2pt, -, >=latex'](0*\w,\y) -- coordinate (x axis) (0,\y) node[right] {}; 
\foreach \x in {0} \draw[ultra thick, black] (\w*\x,\y+0.175) -- (0.3*\x,\y-0.175) node[above] {} ;

\def\y{-1}
\node (b) at (-1,\y) {Proc 1};
\draw[line width=1.2pt, -, >=latex'](0*\w,\y) -- coordinate (x axis) (3*\w,\y) node[right] {}; 
\foreach \x in {3} \draw[ultra thick, black] (\w*\x,\y+0.175) -- (\w*\x,\y-0.175) node[above] {} ;

\def\y{-1.5}
\node (b) at (-1,\y) {Proc 2};
\draw[line width=1.2pt, -, >=latex'](0*\w,\y) -- coordinate (x axis) (6*\w,\y) node[right] {}; 
\foreach \x in {6} \draw[ultra thick, black] (\w*\x,\y+0.175) -- (\w*\x,\y-0.175) node[above] {} ;

\def\y{-2}
\node (b) at (-1,\y) {Proc 3};
\draw[line width=1.2pt, -, >=latex'](3*\w,\y) -- coordinate (x axis) (\w*9,\y) node[right] {}; 
\foreach \x in {9} \draw[ultra thick, black] (\w*\x,\y+0.175) -- (\w*\x,\y-0.175) node[above] {} ;

\def\y{-2.5}
\node (b) at (-1,\y) {Proc 4};
\draw[line width=1.2pt, -, >=latex'](6*\w,\y) -- coordinate (x axis) (12*\w,\y) node[right] {}; 
\foreach \x in {12} \draw[ultra thick, black] (\w*\x,\y+0.175) -- (\w*\x,\y-0.175) node[above] {} ;

\def\y{-3}
\node (b) at (-1,\y) {Proc 5};
\draw[line width=1.2pt, -, >=latex'](9*\w,\y) -- coordinate (x axis) (15*\w,\y) node[right] {}; 
\foreach \x in {15} \draw[ultra thick, black] (\w*\x,\y+0.175) -- (\w*\x,\y-0.175) node[above] {} ;

\def\w{0.2}
\def\xs{4}
\draw[line width=1.2pt, -, >=latex'](\xs+0,0) -- coordinate (x axis) (\xs+\w*17,0) node[right] {}; 
\foreach \x in {0,1,2,...,17} \draw[ultra thick, black] (\xs+\w*\x,0.125) -- (\xs+\w*\x,-0.125) node[below] {} ;
\foreach \x in {0,3,6,9,12,15} \draw[ultra thick, black] (\xs+\w*\x,0.175) -- (\xs+\w*\x,-0.175) node[above] {} ;

\def\y{-0.5}
\draw[fill=gray!40] (\xs-.2,-.25) rectangle (\xs+1.4,-1.75);
\draw[fill=gray!40] (\xs+1.6,-1.75) rectangle (\xs+3.2,-3.25);  

\draw[line width=1.2pt, -, >=latex'](\xs+0*\w,\y) -- coordinate (x axis) (\xs+0,\y) node[right] {}; 
\foreach \x in {0} \draw[ultra thick, black] (\xs+\w*\x,\y+0.175) -- (\xs+0.3*\x,\y-0.175) node[above] {} ;

\def\y{-1}

\draw[line width=1.2pt, -, >=latex'](\xs+0*\w,\y) -- coordinate (x axis) (\xs+3*\w,\y) node[right] {}; 
\foreach \x in {3} \draw[ultra thick, black] (\xs+\w*\x,\y+0.175) -- (\xs+\w*\x,\y-0.175) node[above] {} ;
\foreach \x in {0} \draw[ultra thick, black] (\xs+\w*\x,\y+0.175) -- (\xs+\w*\x,\y-0.175) node[above] {} ;

\def\y{-1.5}

\draw[line width=1.2pt, -, >=latex'](\xs+0*\w,\y) -- coordinate (x axis) (\xs+6*\w,\y) node[right] {}; 
\foreach \x in {6} \draw[ultra thick, black] (\xs+\w*\x,\y+0.175) -- (\xs+\w*\x,\y-0.175) node[above] {} ;
\foreach \x in {0,3} \draw[ultra thick, black] (\xs+\w*\x,\y+0.175) -- (\xs+\w*\x,\y-0.175) node[above] {} ;

\def\y{-2}

\draw[line width=1.2pt, -, >=latex'](\xs+3*\w,\y) -- coordinate (x axis) (\xs+\w*9,\y) node[right] {}; 
\foreach \x in {9} \draw[ultra thick, black] (\xs+\w*\x,\y+0.175) -- (\xs+\w*\x,\y-0.175) node[above] {} ;

\def\y{-2.5}

\draw[line width=1.2pt, -, >=latex'](\xs+6*\w,\y) -- coordinate (x axis) (\xs+12*\w,\y) node[right] {}; 
\foreach \x in {12} \draw[ultra thick, black] (\xs+\w*\x,\y+0.175) -- (\xs+\w*\x,\y-0.175) node[above] {} ;
\foreach \x in {9} \draw[ultra thick, black] (\xs+\w*\x,\y+0.175) -- (\xs+\w*\x,\y-0.175) node[above] {} ;

\def\y{-3}

\draw[line width=1.2pt, -, >=latex'](\xs+9*\w,\y) -- coordinate (x axis) (\xs+15*\w,\y) node[right] {}; 
\foreach \x in {15} \draw[ultra thick, black] (\xs+\w*\x,\y+0.175) -- (\xs+\w*\x,\y-0.175) node[above] {} ;
\foreach \x in {9,12} \draw[ultra thick, black] (\xs+\w*\x,\y+0.175) -- (\xs+\w*\x,\y-0.175) node[above] {} ;

\def\w{0.2}
\def\xs{8.0}
\draw[line width=1.2pt, -, >=latex'](\xs+0,0) -- coordinate (x axis) (\xs+\w*17,0) node[right] {}; 
\foreach \x in {0,1,2,...,17} \draw[ultra thick, black] (\xs+\w*\x,0.125) -- (\xs+\w*\x,-0.125) node[below] {} ;
\foreach \x in {0,3,6,9,12,15} \draw[ultra thick, black] (\xs+\w*\x,0.175) -- (\xs+\w*\x,-0.175) node[above] {} ;

\def\y{-0.5}
\draw[pattern=north east lines] (\xs+-.2,-1.25) rectangle (\xs+1.4,-2.75);

\draw[line width=1.2pt, -, >=latex'](\xs+0*\w,\y) -- coordinate (x axis) (\xs+0,\y) node[right] {}; 
\foreach \x in {0} \draw[ultra thick, black] (\xs+\w*\x,\y+0.175) -- (\xs+0.3*\x,\y-0.175) node[above] {} ;

\def\y{-1}

\draw[line width=1.2pt, -, >=latex'](\xs+0*\w,\y) -- coordinate (x axis) (\xs+3*\w,\y) node[right] {}; 
\foreach \x in {3} \draw[ultra thick, black] (\xs+\w*\x,\y+0.175) -- (\xs+\w*\x,\y-0.175) node[above] {} ;
\foreach \x in {0} \draw[ultra thick, black] (\xs+\w*\x,\y+0.175) -- (\xs+\w*\x,\y-0.175) node[above] {} ;

\def\y{-1.5}

\draw[line width=1.2pt, -, >=latex'](\xs+0*\w,\y) -- coordinate (x axis) (\xs+6*\w,\y) node[right] {}; 
\foreach \x in {6} \draw[ultra thick, black] (\xs+\w*\x,\y+0.175) -- (\xs+\w*\x,\y-0.175) node[above] {} ;
\foreach \x in {0,3} \draw[ultra thick, black] (\xs+\w*\x,\y+0.175) -- (\xs+\w*\x,\y-0.175) node[above] {} ;

\def\y{-2}

\draw[line width=1.2pt, -, >=latex'](\xs+3*\w,\y) -- coordinate (x axis) (\xs+\w*9,\y) node[right] {}; 
\foreach \x in {9} \draw[ultra thick, black] (\xs+\w*\x,\y+0.175) -- (\xs+\w*\x,\y-0.175) node[above] {} ;
\foreach \x in {6,3} \draw[ultra thick, black] (\xs+\w*\x,\y+0.175) -- (\xs+\w*\x,\y-0.175) node[above] {} ;

\def\y{-2.5}

\draw[line width=1.2pt, -, >=latex'](\xs + 6*\w,\y) -- coordinate (x axis) (\xs+12*\w,\y) node[right] {}; 
\foreach \x in {12} \draw[ultra thick, black] (\xs+\w*\x,\y+0.175) -- (\xs+\w*\x,\y-0.175) node[above] {} ;
\foreach \x in {9} \draw[ultra thick, black] (\xs+\w*\x,\y+0.175) -- (\xs+\w*\x,\y-0.175) node[above] {} ;
\foreach \x in {6} \draw[ultra thick, black] (\xs+\w*\x,\y+0.175) -- (\xs+\w*\x,\y-0.175) node[above] {} ;

\def\y{-3}

\draw[line width=1.2pt, -, >=latex'](\xs+9*\w,\y) -- coordinate (x axis) (\xs+15*\w,\y) node[right] {}; 
\foreach \x in {15} \draw[ultra thick, black] (\xs+\w*\x,\y+0.175) -- (\xs+\w*\x,\y-0.175) node[above] {} ;
\foreach \x in {9,12} \draw[ultra thick, black] (\xs+\w*\x,\y+0.175) -- (\xs+\w*\x,\y-0.175) node[above] {} ;
\end{tikzpicture}
\end{adjustbox}

	\caption{\label{fig:communication}Illustration of the communication scheme for $P=6$ processes and distance $k=3$. First (left), each process computes the residual of its $C$-point. In a second step (middle), the residuals are distributed in parallel within the groups of the first decomposition (gray boxes). Last (right), the residuals are distributed within the groups of the second decomposition (shaded boxes), after which each process has all the required residuals.}
\end{figure}
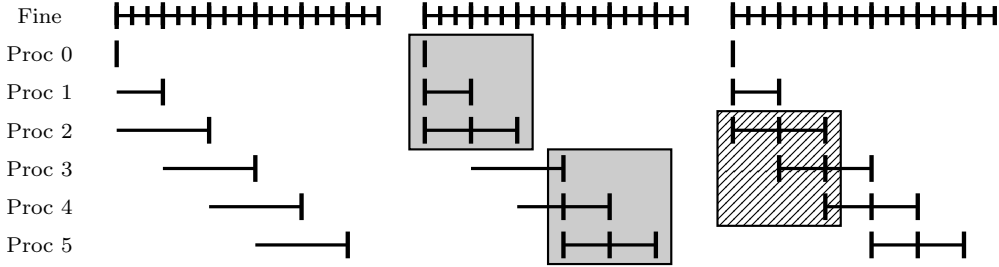

\subsection{Propagation of the initial condition}\label{sec:propagation}

A key feature of the \namealg algorithm is the implicit propagation of the initial condition through the iterations of the method, which allows for using local coarse grids that do not include the initial time point. The idea of implicit propagation of the initial condition is best explained in a two-level example with $F$-relaxation. \Cref{fig:implicit_propagation} shows an example of the distribution of local coarse grids for $N_t=21$, $m=7$ and $k=3$. Only the first three local coarse grids have direct access to the initial condition and, thus, the initial condition is distributed over the first three local coarse grids in the first iteration. All other local coarse grids do not have access to the initial condition at this point. In the second iteration, again only the first three local coarse grids directly contain the initial condition. However, the next two local coarse grids now contain $C$-points of local coarse grids, which directly depend on the initial condition from the previous iteration, making them implicitly depend on the initial condition as well. In the next iteration, the next two local coarse grids implicitly depend on the initial condition, and so on. In the end, the two-level \namealg algorithm with $F$-relaxation requires $\lceil(N_T-1)/(k-1)\rceil$ iterations until all local coarse grids depend implicitly on the initial condition. Note that in the two-level variant with $FCF$-relaxation, the initial value is also propagated to the next $C$-point on the fine grid due to the additional $CF$-relaxation and, thus, the initial condition is propagated faster. 

\begin{figure}
\begin{subfigure}[c]{0.31\textwidth}
\centering
\begin{adjustbox}{width=\textwidth}
  \begin{tikzpicture}[y=1cm, x=1cm, thick, font=\footnotesize]    

\draw[rounded corners,fill=gray!60,gray!60] (-.2,-.25) rectangle (4.2,-1.75);

\def\w{0.2}
\draw[line width=1.2pt, -, >=latex'](0,0) -- coordinate (x axis) (\w*20,0) node[right] {}; 
\foreach \x in {0,1,2,...,20} \draw[ultra thick, black] (\w*\x,0.125) -- (\w*\x,-0.125) node[below] {} ;
\foreach \x in {0,3,6,9,12,15,18} \draw[ultra thick, black] (\w*\x,0.175) -- (\w*\x,-0.175) node[above] {} ;

\def\y{-0.5}

\draw[line width=1.2pt, -, >=latex'](0*\w,\y) -- coordinate (x axis) (0,\y) node[right] {}; 
\foreach \x in {0} \draw[ultra thick, black] (\w*\x,\y+0.175) -- (0.3*\x,\y-0.175) node[above] {} ;

\def\y{-1}

\draw[line width=1.2pt, -, >=latex'](0*\w,\y) -- coordinate (x axis) (3*\w,\y) node[right] {}; 
\foreach \x in {0,3} \draw[ultra thick, black] (\w*\x,\y+0.175) -- (\w*\x,\y-0.175) node[above] {} ;

\def\y{-1.5}

\draw[line width=1.2pt, -, >=latex'](0*\w,\y) -- coordinate (x axis) (6*\w,\y) node[right] {}; 
\foreach \x in {0,3,6} \draw[ultra thick, black] (\w*\x,\y+0.175) -- (\w*\x,\y-0.175) node[above] {} ;

\def\y{-2}

\draw[line width=1.2pt, -, >=latex'](3*\w,\y) -- coordinate (x axis) (\w*9,\y) node[right] {}; 
\foreach \x in {3,6,9} \draw[ultra thick, black] (\w*\x,\y+0.175) -- (\w*\x,\y-0.175) node[above] {} ;

\def\y{-2.5}

\draw[line width=1.2pt, -, >=latex'](6*\w,\y) -- coordinate (x axis) (12*\w,\y) node[right] {}; 
\foreach \x in {6,9,12} \draw[ultra thick, black] (\w*\x,\y+0.175) -- (\w*\x,\y-0.175) node[above] {} ;

\def\y{-3}

\draw[line width=1.2pt, -, >=latex'](9*\w,\y) -- coordinate (x axis) (15*\w,\y) node[right] {}; 
\foreach \x in {9,12,15} \draw[ultra thick, black] (\w*\x,\y+0.175) -- (\w*\x,\y-0.175) node[above] {} ;

\def\y{-3.5}

\draw[line width=1.2pt, -, >=latex'](12*\w,\y) -- coordinate (x axis) (18*\w,\y) node[right] {}; 
\foreach \x in {12,15,18} \draw[ultra thick, black] (\w*\x,\y+0.175) -- (\w*\x,\y-0.175) node[above] {} ;

\end{tikzpicture}
\end{adjustbox}
\subcaption{\label{fig:implicit_propagation1}Iteration 1}

\end{subfigure}
\begin{subfigure}[c]{0.31\textwidth}
\centering
\begin{adjustbox}{width=\textwidth}
  \begin{tikzpicture}[y=1cm, x=1cm, thick, font=\footnotesize]    

\draw[rounded corners,fill=gray!60,gray!60] (-.2,-.25) rectangle (4.2,-2.75);

\def\w{0.2}
\draw[line width=1.2pt, -, >=latex'](0,0) -- coordinate (x axis) (\w*20,0) node[right] {}; 
\foreach \x in {0,1,2,...,20} \draw[ultra thick, black] (\w*\x,0.125) -- (\w*\x,-0.125) node[below] {} ;
\foreach \x in {0,3,6,9,12,15,18} \draw[ultra thick, black] (\w*\x,0.175) -- (\w*\x,-0.175) node[above] {} ;

\def\y{-0.5}

\draw[line width=1.2pt, -, >=latex'](0*\w,\y) -- coordinate (x axis) (0,\y) node[right] {}; 
\foreach \x in {0} \draw[ultra thick, black] (\w*\x,\y+0.175) -- (0.3*\x,\y-0.175) node[above] {} ;

\def\y{-1}

\draw[line width=1.2pt, -, >=latex'](0*\w,\y) -- coordinate (x axis) (3*\w,\y) node[right] {}; 
\foreach \x in {0,3} \draw[ultra thick, black] (\w*\x,\y+0.175) -- (\w*\x,\y-0.175) node[above] {} ;

\def\y{-1.5}

\draw[line width=1.2pt, -, >=latex'](0*\w,\y) -- coordinate (x axis) (6*\w,\y) node[right] {}; 
\foreach \x in {0,3,6} \draw[ultra thick, black] (\w*\x,\y+0.175) -- (\w*\x,\y-0.175) node[above] {} ;

\def\y{-2}

\draw[line width=1.2pt, -, >=latex'](3*\w,\y) -- coordinate (x axis) (\w*9,\y) node[right] {}; 
\foreach \x in {3,6,9} \draw[ultra thick, black] (\w*\x,\y+0.175) -- (\w*\x,\y-0.175) node[above] {} ;

\def\y{-2.5}

\draw[line width=1.2pt, -, >=latex'](6*\w,\y) -- coordinate (x axis) (12*\w,\y) node[right] {}; 
\foreach \x in {6,9,12} \draw[ultra thick, black] (\w*\x,\y+0.175) -- (\w*\x,\y-0.175) node[above] {} ;

\def\y{-3}

\draw[line width=1.2pt, -, >=latex'](9*\w,\y) -- coordinate (x axis) (15*\w,\y) node[right] {}; 
\foreach \x in {9,12,15} \draw[ultra thick, black] (\w*\x,\y+0.175) -- (\w*\x,\y-0.175) node[above] {} ;

\def\y{-3.5}

\draw[line width=1.2pt, -, >=latex'](12*\w,\y) -- coordinate (x axis) (18*\w,\y) node[right] {}; 
\foreach \x in {12,15,18} \draw[ultra thick, black] (\w*\x,\y+0.175) -- (\w*\x,\y-0.175) node[above] {} ;

\end{tikzpicture}
\end{adjustbox}
\subcaption{\label{fig:implicit_propagation2} Iteration 2}

\end{subfigure}
\begin{subfigure}[c]{0.31\textwidth}
\centering
\begin{adjustbox}{width=\textwidth}
\begin{tikzpicture}[y=1cm, x=1cm, thick, font=\footnotesize]    

\draw[rounded corners,fill=gray!60,gray!60] (-.2,-.25) rectangle (4.2,-3.75);

\def\w{0.2}
\draw[line width=1.2pt, -, >=latex'](0,0) -- coordinate (x axis) (\w*20,0) node[right] {}; 
\foreach \x in {0,1,2,...,20} \draw[ultra thick, black] (\w*\x,0.125) -- (\w*\x,-0.125) node[below] {} ;
\foreach \x in {0,3,6,9,12,15,18} \draw[ultra thick, black] (\w*\x,0.175) -- (\w*\x,-0.175) node[above] {} ;

\def\y{-0.5}

\draw[line width=1.2pt, -, >=latex'](0*\w,\y) -- coordinate (x axis) (0,\y) node[right] {}; 
\foreach \x in {0} \draw[ultra thick, black] (\w*\x,\y+0.175) -- (0.3*\x,\y-0.175) node[above] {} ;

\def\y{-1}

\draw[line width=1.2pt, -, >=latex'](0*\w,\y) -- coordinate (x axis) (3*\w,\y) node[right] {}; 
\foreach \x in {0,3} \draw[ultra thick, black] (\w*\x,\y+0.175) -- (\w*\x,\y-0.175) node[above] {} ;

\def\y{-1.5}

\draw[line width=1.2pt, -, >=latex'](0*\w,\y) -- coordinate (x axis) (6*\w,\y) node[right] {}; 
\foreach \x in {0,3,6} \draw[ultra thick, black] (\w*\x,\y+0.175) -- (\w*\x,\y-0.175) node[above] {} ;

\def\y{-2}

\draw[line width=1.2pt, -, >=latex'](3*\w,\y) -- coordinate (x axis) (\w*9,\y) node[right] {}; 
\foreach \x in {3,6,9} \draw[ultra thick, black] (\w*\x,\y+0.175) -- (\w*\x,\y-0.175) node[above] {} ;

\def\y{-2.5}

\draw[line width=1.2pt, -, >=latex'](6*\w,\y) -- coordinate (x axis) (12*\w,\y) node[right] {}; 
\foreach \x in {6,9,12} \draw[ultra thick, black] (\w*\x,\y+0.175) -- (\w*\x,\y-0.175) node[above] {} ;

\def\y{-3}

\draw[line width=1.2pt, -, >=latex'](9*\w,\y) -- coordinate (x axis) (15*\w,\y) node[right] {}; 
\foreach \x in {9,12,15} \draw[ultra thick, black] (\w*\x,\y+0.175) -- (\w*\x,\y-0.175) node[above] {} ;

\def\y{-3.5}

\draw[line width=1.2pt, -, >=latex'](12*\w,\y) -- coordinate (x axis) (18*\w,\y) node[right] {}; 
\foreach \x in {12,15,18} \draw[ultra thick, black] (\w*\x,\y+0.175) -- (\w*\x,\y-0.175) node[above] {} ;

\end{tikzpicture}
\end{adjustbox}
\subcaption{\label{fig:implicit_propagation3} Iteration 3}

\end{subfigure}
	\caption{\label{fig:implicit_propagation} Implicit propagation (gray box) of the initial condition for a two-level \namealg variant with $F$-relaxation and $N_t=21$, $m=3$ and $k=3$.
}
\end{figure}
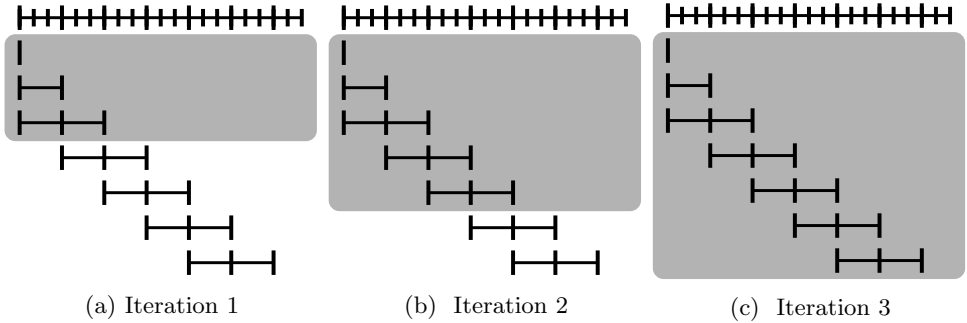

\subsection{Size of local coarse grids}\label{sec:heat_equation}

The parameter $k$ defines the size of the local coarse grids and, thus, the number of sequential solves needed on the coarse grid. In the following, we consider the influence of the parameter $k$ on the convergence of \namealg applied to a standard model problem for parallel-in-time methods, the one-dimensional heat equation,
\begin{equation} \label{eq:heat_equation}
u_t - u_{xx} = b(x,t) \;\; \text{ in } \; [0,3]\times[0,\pi],
\end{equation}
subject to the initial condition $u(x,0) = \sin(\pi x)$ and homogeneous Dirichlet boundary conditions in space. The forcing term is choosen as $b(x,t)=-\sin(\pi x) (\sin(t) - \pi^2 \cos(t))$, such that the exact solution is given by $u(x,t) = \sin(\pi x)\cos(t)$.

We discretize \eqref{eq:heat_equation} using second-order central finite differences with $1{,}025$ degrees of freedom in space and on an equidistant time grid with $16{,}384$ time points using backward Euler. We investigate the behavior of the \namealg algorithm for the two-level case, and choose different coarsening factors $m$ and distances $k$. We restrict ourselves to the two-level case, since we want to study the effect of using local coarse grids of various sizes $k$. For all simulations, the stopping tolerance is based on the discrete 2-norm of the absolute space-time residual with a tolerance of $10^{-7}$ and a random initial guess is chosen for all time points except for the initial condition. This choice guarantees that no knowledge of the right-hand side is used that could affect the convergence. Note that this is only a good choice for investigating the behavior of the algorithm and is not recommended in practice.

\Cref{fig:heat_diff_k-l} shows the required number of iterations to reach the stopping criterion for a two-level \namealg variant with $F$-relaxation and a coarsening factor of $m=128$ as a function of size $k$. 
Note that while the variant with $k=128$, which is equivalent to Parareal, performs $127$ sequential time integrations on the coarse level, equivalent convergence can be obtained with $k=12$, $10\times$ less coarse-grid solves. 
\Cref{fig:heat_diff_k-r} presents iterations to convergence as a function of \emph{the ratio} of local to global coarse-grid points. For three different coarsening factors, we see that convergence does not improve beyond the same ratio of $k/$(\#C-points), in this case about $0.08$. Although this parameter is likely problem specific, \Cref{fig:heat_diff_k-r} does suggest the choice of $k$ is relatively agnostic to coarsening factor by posing it relative to the global coarse-grid size.

\begin{figure}[ht!]
  \centering
	\begin{subfigure}[t]{0.475\textwidth}
		\includegraphics[width=0.96\linewidth]{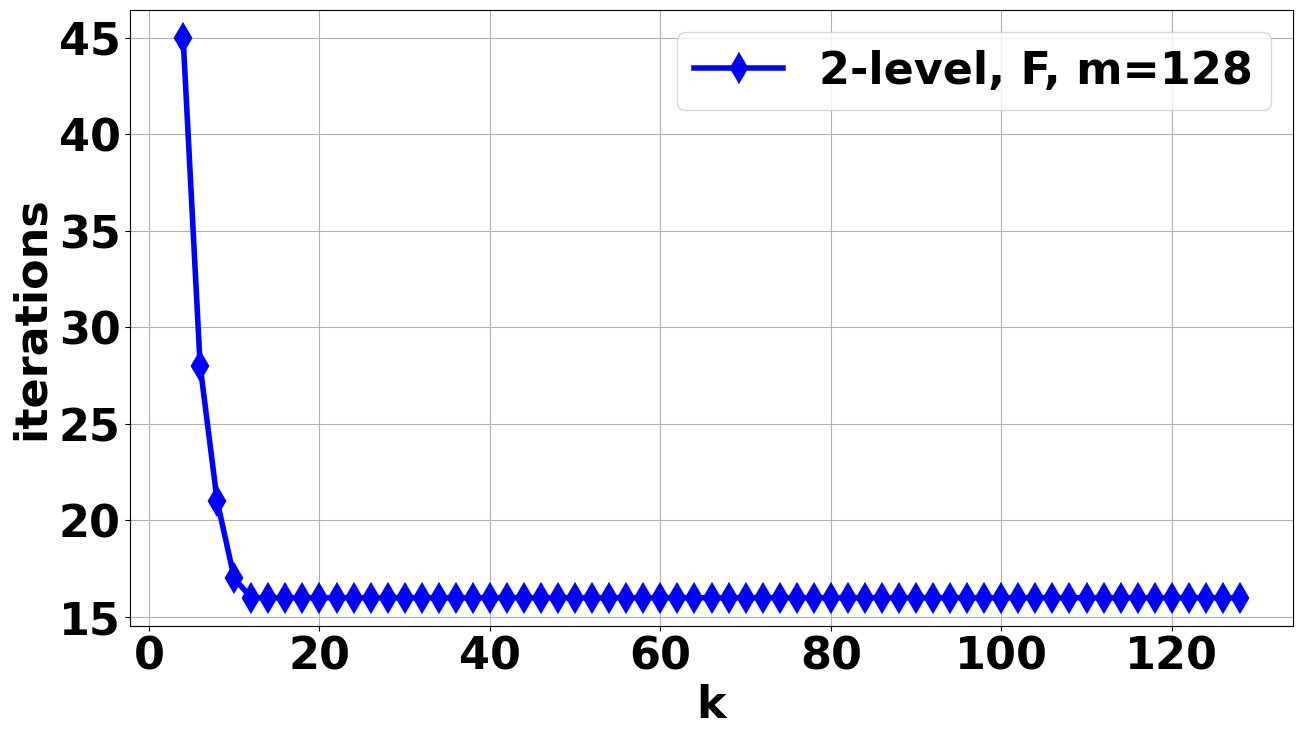}
		\caption{Iterations to convergence as a function of $k$.}
		\label{fig:heat_diff_k-l}
  \end{subfigure}\quad
	\begin{subfigure}[t]{0.475\textwidth}
		\includegraphics[width=\linewidth]{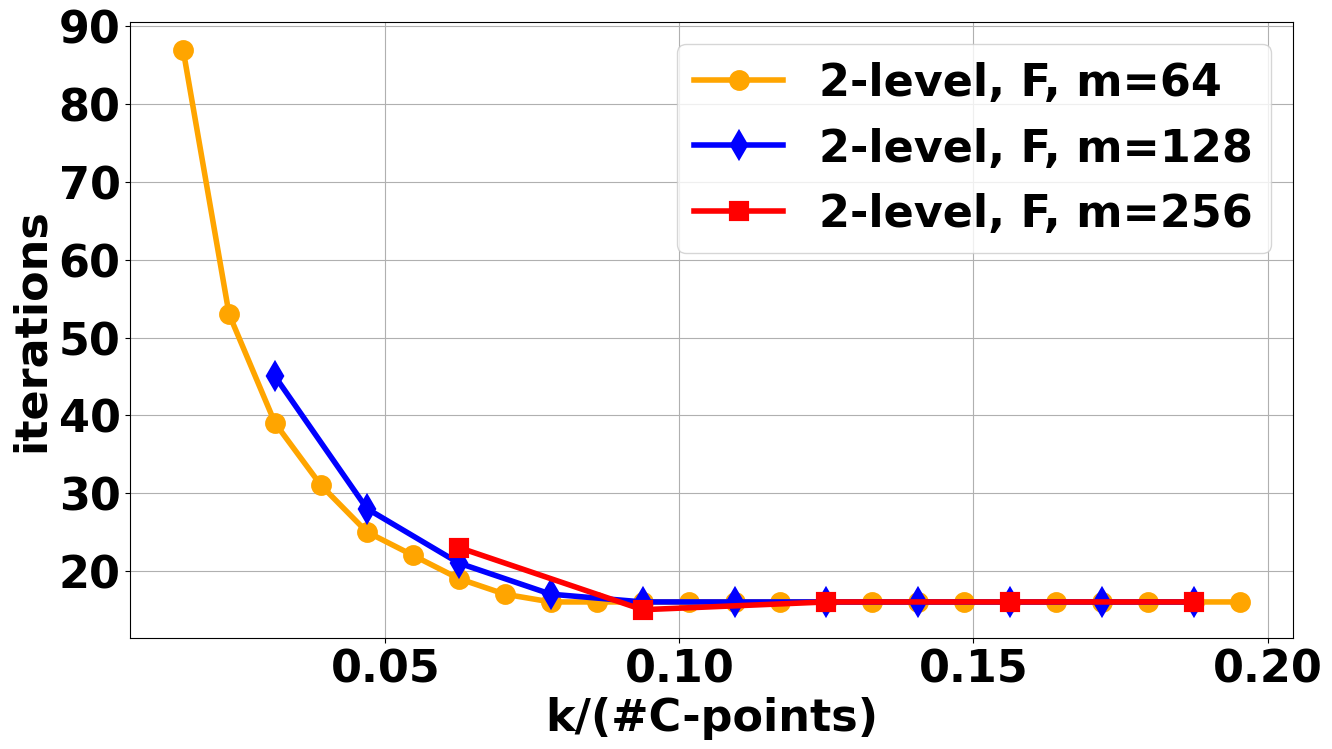}
		\caption{Iterations to convergence as a function of $k$ divided by
		the number of C-points.}
		\label{fig:heat_diff_k-r}
  \end{subfigure}
  \caption{{Required iterations for \namealg variants for the 1D heat equation.}}
\end{figure}

\Cref{fig:heat_conv-l} plots convergence behavior for two-level \namealg with coarsening factor $m=128$ and various choices for $k$. The variant with $k=128$ (i.\,e., Parareal) has uniform convergence behavior, while convergence for smaller $k$ is split into three parts. Initially, convergence is slower than for $k=128$. The smaller $k$ is chosen, the slower is the convergence. After a few iterations, once the initial condition has been implicitly propagated over all local coarse grids, a drastic improvement in convergence can be seen for all three variants. This lasts for a few iterations, until convergence rates then asymptote to that of Parareal. \Cref{fig:heat_conv-r} shows the $\ell^\infty$-norm of the error for the four variants and the time-stepping method, verifying that all variants compute the same solution up to a certain tolerance. For all variants, the norm of the error to the exact solution has the same order of magnitude.

\begin{figure}[hbt!]
	  \centering
	\begin{subfigure}[t]{0.475\textwidth}
		\includegraphics[width=\linewidth]{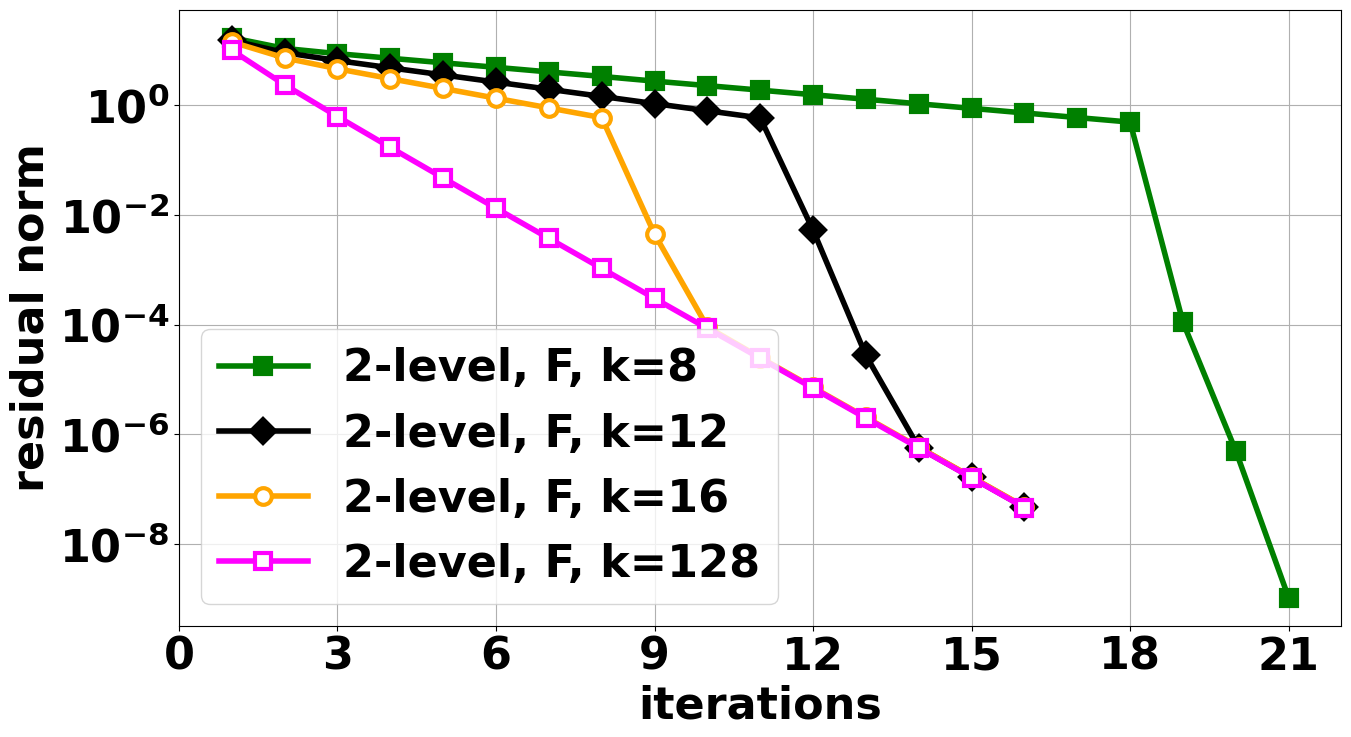}
		\caption{Residual norm as a function of iteration for two-level \namealg with coarsening factor $m=128$.}
		\label{fig:heat_conv-l}
  \end{subfigure}\quad
	\begin{subfigure}[t]{0.475\textwidth}		
	\includegraphics[width=\linewidth]{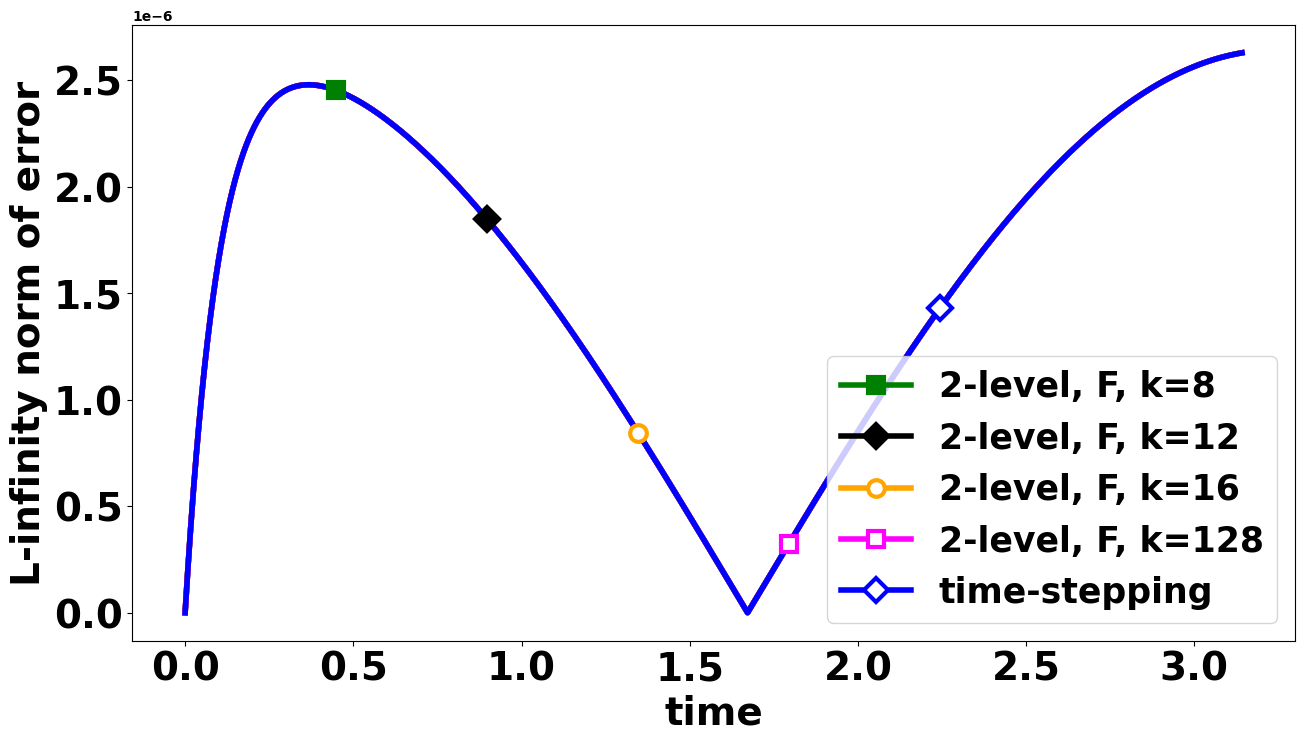}
	\caption{$\ell^\infty$-norm of error with respect to the exact solution at each point in time.}
	\label{fig:heat_conv-r}
  \end{subfigure}
  \caption{{Results of \namealg variants in terms of residual norm and error.}}
\end{figure}

\section{Parallel results}\label{sec:parallel_experiments}
In this section, we present numerical results for \namealg applied to two challenging, nonlinear time-dependent problems: the 2D Gray Scott example of a chemical reaction of two substances, and a realistic model of an electrical machine. In addition, we apply different two- and three-level variants of \namealg and compare runtimes and iteration counts with the corresponding variants of Parareal and MGRIT, respectively. For the two-level variants, we apply $F$-relaxation and we choose the coarsening factor such that the number of coarse-grid points is equal to the number of processes used for the simulation enabling perfect parallelization on the fine level. For the three-level algorithms, we apply non-uniform coarsening strategies with a coarsening factor of $m_0=64$ between the fine and the intermediate level, and different factors between the intermediate and the coarse level. For the \namealg algorithm, we additionally choose $k=N_t^{(1)}/2$ or $k=N_t^{(2)}/2$, respectively, i.\,e., we halve the number of grid points on the coarsest level to be computed serially. 

All simulations were performed on an Intel Xeon Phi Cluster consisting of four 1.4 GHz Intel Xeon Phi processors. The code for all experiments can be found in the \texttt{PyMGRIT} repository \cite{pymgrit-github}, and this package is also used for simulations with Parareal and MGRIT. For all experiments, we use all possible resources for temporal parallelization, i.\,e., we do not use spatial parallelization (largely due to limited resources). For a brief discussion on the effect of spatial parallelism for the different algorithms, see \Cref{app:AppendixC}.

\subsection{Gray Scott}\label{sec:parallel_experiments:gray_scott}
We consider the 2D Gray-Scott problem \cite{Pearson189} of a chemical reaction of two components $\mathcal{U}$ and $\mathcal{V}$, given by
\begin{subequations}
\begin{align*}
u_t &= D_u\Delta u - uv^2+F(1-u), \\
v_t &= D_v\Delta v + uv^2 - (K+F)u,
\end{align*}
\end{subequations}
where $u=u(x,y,t)$ and $v=v(x,y,t)$ are the concentration of $\mathcal{U}$ and $\mathcal{V}$, respectively, $D_u$ and $D_v$ are the diffusion rates, $F$ is the feed rate, and $K$ is the removal rate. For our simulations, we choose the spatial domain $[0,2.5]^2$ with periodic boundary conditions, and the time interval $[0,256]$. Further, we choose the parameters $F=0.024$, $K=0.06$, $D_u=8\times10^{-5}$, and $D_v=4\times10^{-5}$, and we consider the initial value
\begin{subequations}
\begin{align*}
u(x,y,0) &= 1 - 2 (0.25\sin(4\pi x)^2\sin(4\pi y)^2), && (x, y)\in [1,1.5]^2\\
v(x,y,0) &= 0.25\sin(4\pi x)^2 \sin(4\pi y)^2, && (x, y)\in [1,1.5]^2,
\end{align*}
\end{subequations}
and $u(x,y,0)=1$ and $v(x,y,0)=0$ otherwise. The problem is discretized using central finite differences with $128^2$ points in space and on an equidistant time grid with $16{,}384$ points using backward Euler. We solve the resulting nonlinear problem using Newton's method of \texttt{PETSc} \cite{petsc-user-ref} with a relative and absolute tolerance of $10^{-10}$.

We apply two-level and three-level \namealg and Parareal variants with nested iterations to compute an improved initial guess. In the nested iteration strategy, Parareal and MGRIT solve the global coarse-grid problem at the coarsest level, while \namealg uses the local coarse grids instead of the global grid. The stopping criterion for all variants is based on the discrete 2-norm of the space-time residual with a tolerance of $10^{-7}$. 

\begin{table*}
	\renewcommand{\arraystretch}{1.3}
	\begin{center}
	\begin{tabular}{ |l|c|c|c|c|c|c|c|c|}
		\hline
		\multirow{3}{*}{Method} & \multirow{3}{*}{$m$} & \multirow{3}{*}{$k$} & \multirow{3}{*}{\# Procs} & \multirow{3}{*}{\# Iters} & \multirow{3}{*}{\parbox{1.3cm}{\centering Setup\\ time}} & \multirow{3}{*}{\parbox{1.4cm}{\centering Solve \\ time}} & \multirow{3}{*}{\parbox{1.5cm}{\centering Speedup w.r.t\\ Parareal}} 

		\\
		 	   &            &              &            &            &            &  &\\
		 	   		 	   &            &              &            &            & &            & 

		 	    \\ \hline \hline
		\multirow{4}{*}{Parareal}& 512 &- & 32 & 12 & 1{,}338 s & 172{,}068 s & - 

		\\ \cline{2-8}
		& 256 & - & 64 & 10 & 2{,}308 s & 89{,}288 s &-

		\\ \cline{2-8} 
		& 128 &- & 128 & 9 & 3{,}958 s & 66{,}485 s &-

		\\ \cline{2-8}
		& 64 &- & 256 & 7 & 7{,}646 s & 66{,}272 s  &-

		\\ \hline \hline
		\multirow{4}{*}{\parbox{1.85cm}{Two-level\\ \namealgnospace}}& 512&16 & 32 & 12 & 701 s & 165{,}351 s & 1.04 

		\\ \cline{2-8} 
		& 256 & 32 & 64 & 10 & 1{,}167 s & 78{,}230 s & 1.15 

		\\ \cline{2-8}
		& 128&64 & 128 & 9 & 2{,}022 s & 48{,}675 s & 1.39 

		\\ \cline{2-8}
		& 64&128 & 256 & 7 & 3{,}812 s & 39{,}895 s & 1.69  

		\\ \hline	
		\end{tabular}
	\end{center}
	\caption{Iteration counts, setup times (for computing an improved initial guess), and runtimes of the solve phase of two-level \namealg and Parareal variants applied to the 2D Gray-Scott problem for various numbers of processes.}
	\label{tab:gray_two_level}
\end{table*}

\Cref{tab:gray_two_level} shows the number of iterations and runtimes of the setup and solve phases of two-level \namealg and Parareal variants. The results show that iteration counts of \namealg are equal to iteration counts of Parareal with the same coarsening strategy. Furthermore, a finer coarse grid significantly reduces the number of iterations required. While 12 iterations are needed for the two-level variants with a coarse grid of only 32 points, this number is reduced to seven iterations for the variants with 256 coarse-grid points. However, this reduction in iterations is accompanied by significantly more expensive sequential coarse-grid solves, reflected in increasing setup times with increasing points on the coarse grid. However, if the number of points on the coarse grid doubles, the setup time does not double. This is because a smaller time step requires fewer Newton iterations and, thus, affects the duration of the application of each time integration. The setup time of each \namealg variant is about half as long as that of the corresponding Parareal variant due to the choice of $k$.  
Looking at the runtimes of the solve phase, we see that \namealg is always faster than the corresponding Parareal variant, achieving a speedup of up to a factor of 1.69 compared to Parareal. Furthermore, we see that while the Parareal algorithm does not scale for more than 128 processes, since the serial part of the algorithm dominates the benefit of the additional parallelization of the fine level, the \namealg algorithm shows good parallel scaling up to 256 processes.

\begin{table*}
	\renewcommand{\arraystretch}{1.3}
	\begin{center}
	\begin{tabular}{ |l|c|c|c|c|c|c|c|c|}
		\hline
		\multirow{2}{*}{Method} & \multirow{2}{*}{$m$} & \multirow{2}{*}{$k$}  & \multirow{2}{*}{\# Iters}& \multirow{2}{*}{\parbox{1.5cm}{\centering Setup\\ time}} & \multirow{2}{*}{\parbox{1.5cm}{\centering Solve\\ time}} & \multirow{2}{*}{\parbox{2.2cm}{\centering Speedup w.r.t\\ MGRIT}} 

		 \\

		&&&&&& \\ \hline \hline
		\multirow{4}{*}{MGRIT}& (64,16) &-  & 7 & 3{,}525 s & 43{,}604 s & -

		\\ \cline{2-7}
		& (64,8)&- & 7  & 3{,}498 s & 38{,}420 s & - 

		\\ \cline{2-7}
		& (64,4)&-& 7  & 4{,}980 s & 42{,}285 s & - 

		\\ \cline{2-7}
		& (64,2)&- & 6  & 8{,}075 s & 45{,}131 s & - 

		 \\ \hline \hline
		\multirow{4}{*}{\namealgnospace}& (64,16)& 8  & 7 & 2{,}864 s & 41{,}054 s & 1.07 

		\\ \cline{2-7} 
		& (64,8)& 16  & 7 & 2{,}174 s & 33{,}688 s & 1.17

		\\ \cline{2-7}
		& (64,4)& 32 & 7 & 2{,}713 s & 34{,}063 s & 1.29

		\\ \cline{2-7}
		& (64,2)& 64  & 7 & 4{,}247 s & 38{,}571 s & 1.24 

		\\ \hline	
		\end{tabular}
	\end{center}
	\caption{Iteration counts and runtimes of the setup and solve phase on 256 processes of three-level \namealg and MGRIT variants with $FCF$-relaxation and different non-uniform coarsening strategies applied to the 2D Gray-Scott problem.}
	\label{tab:gray_three_level}
\end{table*}

\Cref{tab:gray_three_level} presents similar results to \Cref{tab:gray_two_level} for four different three-level variants of \namealg and MGRIT with $FCF$-relaxation on 256 processes.
The number of iterations here does not depend as much on the coarsest grid as in the two-level case, but we still see that the MGRIT variant with the coarsening strategy $(64,2)$, i.\,e., the variant with the most points on the second level, requires the fewest iterations. The corresponding \namealg variant needs one additional iteration, but after the sixth iteration the stopping criterion is slightly missed. A minimal increase in $k$ would likely eliminate this extra iteration. 
In terms of solve times, we see that all variants of the \namealg algorithm are faster than the corresponding MGRIT variants, even the variant that requires an additional iteration. Again, the more points on the coarsest level, the higher the speedup of \namealg over MGRIT.

\begin{figure}
	\centerline{\includegraphics[width=.8\linewidth]{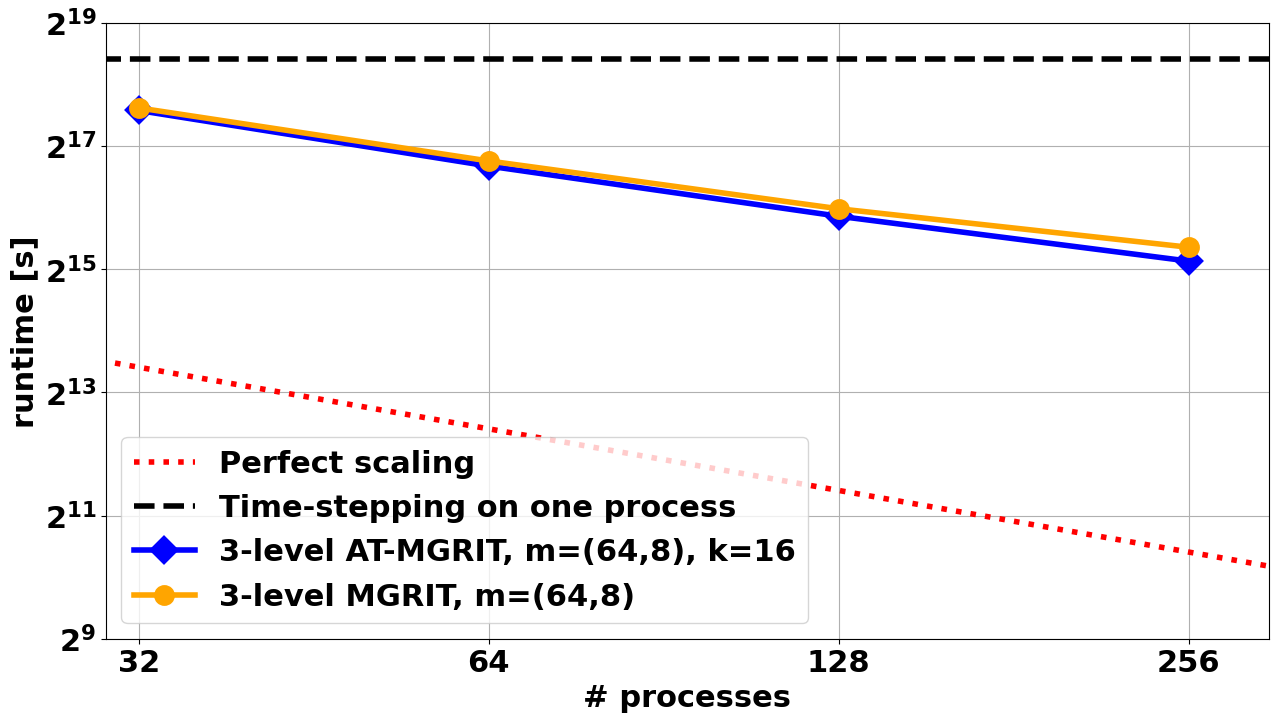}}
	\caption{Strong scaling results for one three-level \namealg variant, {the corresponding MGRIT variant,} and sequential time-stepping on one process applied to the 2D Gray-Scott problem. The red dotted line indicates the perfect scaling based on the runtime of time-stepping.}
	\label{fig:gray_scaling}
\end{figure}

\Cref{fig:gray_scaling} shows the overall runtime (setup and solve) for one \namealg variant (blue line) {and the corresponding MGRIT variant (orange line)} as a function of the number of processes and the runtime of time-stepping using only one process (black dashed line) which is about four days. For reference, the red dotted line indicates the behavior of perfect scaling based on the runtime of time-stepping. While the runtime almost halves when using 32 and 64 processes, the scaling curve starts to flatten slightly with a higher number of processes. This is mainly because only the fine level computations have an additional benefit from more processes due to the chosen coarsening strategy, and the runtime of the coarser levels becomes more and more dominant. {However, compared to the corresponding MGRIT variant, the \namealg variant scales better due to its reduced work at the coarsest level.}

\subsection{Induction machine}\label{sec:parallel_experiments:electrical_machine}

The standard approach for the simulation of electrical machines is based on neglecting the displacement current in Maxwell's equations \cite{jackson_classical_1998}. The resulting magnetoquasistatic approximation, the so-called eddy current problem, is defined in terms of the unknown magnetic vector potential $\mathbf{A}: \Omega \times (t_0,t_f] \rightarrow \mathbb{R}^3$ as 
\begin{subequations}\label{eq:system_eddy_current}
\begin{alignat}{2}
		\sigma \partial_t \mathbf{A} + \nabla \times \left(\nu(\cdot) \nabla\times\mathbf{A}\right)&= \mathbf{J}_{s} \; &&\text{in} \; \Omega \times (t_0,t_f], \nonumber \\ 
		\mathbf{n} \times \mathbf{A} &= 0 &&\text{on} \; \partial \Omega, \nonumber  
\end{alignat}
\end{subequations}
with initial value $\mathbf{A}\left(\mathbf{x},t_0\right) = \mathbf{A}_0(\mathbf{x})$, spatial domain $\Omega$, consisting of the rotor, stator, and the air gap in between, and where $(t_0,t_f]$ is the time interval. The magnetic flux $\mathbf{B}=\nabla\times\mathbf{A}$ vanishes at the boundaries $\partial \Omega$ of the spatial domain (Dirichlet boundary condition). Three ($n_s = 3$) distributed stranded conductors are modeled by the source current density
$
		\mathbf{J}_{s} = \sum_{s=1}^{n_{s}} \mathbf{\chi}_s i_s,
$
with winding functions $\mathbf {\chi}_s: \Omega \rightarrow \mathbb{R}^3$ and currents $i_s: \timeint \rightarrow \mathbb{R}^3$. An attached electrical network provides a connection between the so-called flux-linkage, i.\,e., the spatially integrated time derivative of the magnetic vector potential, and the voltage. The scalar electrical conductivity $\sigma(\mathbf{x}) \geq 0$ and the (nonlinear) magnetic reluctivity $\nu(\mathbf{x},|\nabla\times\mathbf{A}|) > 0$ encode the geometry. To consider the rotation of the rotor, the problem is extended by an additional equation of motion; we refer to \cite{Bolten_etal_2020} for more details.

In the following, we consider a cross-section in the $x,y$-plane to reduce the spatial domain to a two-dimensional ($2$D) domain $\Omega_{2D} \subset \mathbb{R}^2$. Discretizing 
in space using finite elements with $n_a$ degrees of freedom yields a system of differential-algebraic equations of the form
\begin{subequations}\label{eq:semi_discrete_system}
	\begin{align}
		M \mathbf{u}'(t) + K(\mathbf{u}(t)) \mathbf{u}(t) &= \mathbf{f}(t), \quad t \in \timeint \label{eq:semi_discrete_system_a}\\
		\mathbf{u}(t_0) &=\mathbf{u}_0, \label{eq:semi_discrete_system_b}
	\end{align}
\end{subequations}
with unknown $\mathbf{u}^\top=[\mathbf{a}^\top, \mathbf{i}^\top, \theta, \omega] : \timeint \rightarrow \mathbb{R}^n$. At each point in time $t \in (t_0,t_f]$, the solution $\mathbf{u}(t)\in\mathbb{R}^n$ consists of the magnetic vector potentials $\mathbf{a}\left(t\right) \in \mathbb{R}^{n_a}$, the currents of the three phases $\mathbf{i}\left(t\right) \in \mathbb{R}^3$, the rotor angle $\theta\left(t\right) \in \mathbb{R}$, and the angular velocity of the rotor $\omega\left(t\right) \in \mathbb{R}$. The given voltages $\mathbf{v}(t) \in \mathbb{R}^3$ and the mechanical excitation define the right-hand side $\mathbf{f}(t)$. Note, that \eqref{eq:semi_discrete_system} is a differential-algebraic equation of index-1 \cite{Bartel_2011aa, Cortes-Garcia_2019aa} due to the presence of non-conducting materials, i.\,e., regions with $\sigma = 0$, in the domain, which can be treaten by standard techniques in a parallel-in-time setting \cite{Schops_2018aa}.  

The multi-slice finite element model ``im\_$3$\_kw'' \cite{JGyselinck2001_multi_slice} of a four-pole squirrel cage induction machine is used for modeling the semi-discrete problem \eqref{eq:semi_discrete_system}. A mesh representation with $n_a=4449$ degrees of freedom is generated using Gmsh \cite{gmsh, gmsh_website}. Further, we choose $t_0=0$, $t_f=0.2$ and a time grid with $N_t=16{,}384$ points, which corresponds to a time-step size of $\delta t \approx 10^{-5}$. Note that this time interval is required approximately to reach the steady-state of the machine. For the simulations, only a quarter of the machine is considered with periodic boundary conditions and the \texttt{GetDP} library \cite{PDular_etal_1998,getdp}, which implements the time integration using backward Euler, is used for the computations. At each time step, the \texttt{GetDP} library is called and a nonlinear problem is solved by applying Newton's method with damping. For the stopping criterion of the Newton solver, we choose a relative error of $10^{-6}$. The machine is supplied by a three-phase sinusoidal voltage source, and, as proposed in \cite{JGyselinck2001_multi_slice}, an initial ramp-up of the applied voltage is used for reducing the transient behavior of the motor for the time interval $[0,0.04]$.

In the following, we present results for one Parareal variant and several two-level \namealg variants. For all experiments, we use an improved initial guess given by a global coarse-grid solve. Unfortunately, the use of too large time steps on the coarse level in the simulation of the electrical machine in the time-parallel setting causes at least one nonlinear solve within \texttt{GetDP} to fail to converge. To overcome this problem, we apply subcycling at the coarse level, i.\,e., we apply three smaller steps per time step at the coarse level, reducing the time step size and improving the accuracy of the solution. For all algorithm, we use a convergence criterion based on the relative change in Joule losses, an important quantity of an electrical machine, at all $C$-points of two successive iterations; see \cite{Bolten_etal_2020} for details. The algorithm terminates when the maximum norm of the relative difference of two successive iterations is less than $1\%$. Note that this criterion does not guarantee convergence to the discrete time-stepping solution, but for each variant it has been verified that it does indeed iterate to the discrete time-stepping solution.

\begin{table*}
	\renewcommand{\arraystretch}{1.3}
	\begin{center}
	\begin{tabular}{ |l|c|c|c|c|c|c|c|c|}
		\hline		\multirow{3}{*}{Method} & \multirow{3}{*}{$k$} & \multirow{3}{*}{Iterations} & \multirow{3}{*}{\parbox{2.5cm}{\centering Total time\\ (Setup + Solve)}} & \multirow{3}{*}{\parbox{2.0cm}{\centering Speedup w.r.t\\ Parareal}} & \multirow{3}{*}{\parbox{2.2cm}{Speedup w.r.t \\ time-stepping \\ on one process}}  \\
		 	   &            &              &            &            &     \\
		 	   		 	   &            &              &            &            &         
		 	    \\ \hline \hline
		\multirow{1}{*}{Parareal}& - & 5 & 40{,}544 s & - & 4.64 \\ \hline \hline
		\multirow{5}{*}{\parbox{1.85cm}{Two-level\\ \namealgnospace}}
		& 16 & 6 &  32{,}710 s & 1.24 & 5.75  \\ \cline{2-6} 
		& 18 & 6 &  33{,}337 s & 1.22 & 5.64  \\ \cline{2-6} 		
		& 20 & 6 &  33{,}996 s & 1.19 & 5.53  \\ \cline{2-6} 
		& 22 & 6 & 34{,}626 s & 1.17 & 5.43  \\ \cline{2-6}
		& 24 & 5 & 30{,}582 s & 1.33 & 6.15  \\ \hline
		\end{tabular}
	\end{center}
\caption{Iteration counts and total runtimes on 64 processes of Parareal and various two-level \namealg variants with a coarsening factor of 256 for the simulation of the induction machine.}
	\label{tab:results_nonlin2}
\end{table*}

\Cref{tab:results_nonlin2} shows the number of iterations, total runtimes, and the speedup compared to sequential time-stepping on one process for different \namealg variants and Parareal. Furthermore, the speedup compared to Parareal is shown for all \namealg variants. Comparing the number of iterations, the Parareal algorithm and \namealg with $k=24$ both require five iterations to convergence. For smaller $k$, six iterations are needed to reach the stopping criterion. Despite the increased number of iterations for some variants, the total runtime of all \namealg variants is smaller than that of Parareal, with the largest speedup of a factor of approximately 1.33. Note that the time for the setup phase is the same for all variants and is about $2{,}891$ s. Note also that both algorithms treat the fine level identically, and the improvement comes only from using local coarse grids instead of one global coarse grid. For comparison, the simulation time using serial time-stepping on one process is $188{,}123$ s, which is more than two days. The fastest \namealg variant needs less than nine hours, which corresponds to a speedup of a factor of 6.15.

\section{Conclusion}\label{sec:conclusion}
In this paper, we introduce the new \namealg algorithm, an iterative parallel-in-time algorithm for solving time-dependent problems. While the fine level(s) are treated as in the Parareal/MGRIT algorithm, the \namealg algorithm modifies the coarsest level computations. Instead of considering one global time grid that covers the entire time interval and is solved sequentially at the coarsest level, the \namealg algorithm uses a number of truncated, overlapping local coarse grids, one for each point on the global coarse grid. Each of these time grids is independent and covers only a fraction of the global time interval, allowing each problem to be solved simultaneously and reducing the serial work of the algorithm at the coarsest level.

Theoretical and numerical investigations of the algorithm show that the use of local coarse grids, which are not all connected to the initial value of the problem, introduces an additional error term compared to classical Parareal/MGRIT, which may affect the convergence at the beginning of the algorithm. However, the \namealg algorithm takes advantage of its iterative nature and eliminates this additional error term during several iterations, achieving convergence in the same number of iterations as with Parareal/MGRIT while significantly reducing the serial cost on the coarse level. Simulation of challenging nonlinear problems shows that the MGRIT algorithm can provide significant speedup compared to Parareal/MGRIT. Future work involves implementing and studying \namealg on GPUs and emerging shared-memory computing architectures, where the local and asynchronous aspect of coarse grid solves is likely to be particularly advantageous.

\appendix
\section{Error propagation for ideal local coarse problems}\label{app:AppendixA}
We consider error-propagation $\mathcal{E}_e$ for one $C$-point $p=0,\hdots,N_T$ using the ideal local coarse-grid problem (i.\,e., exact coarse grid and inverses). The structure of the matrices differs for the first $k$ $C$-points from all other $C$-points, since the local coarse grids corresponding to the first $k$ $C$-points contain all $C$-points prior in time. Here, we want to study the effect of local coarse grids that do not extend back to $t=0$. Therefore, we start by considering all local coarse grids $p\geq k$ and subsequently discuss the structure for $p<k$. For $p\geq k$ the matrix $R_I^{(p)}A$ is given by
\begin{center}
\begin{equation}\label{eq:a1}
\begin{tikzpicture}[
baseline={-0.5ex}, 
mymatrixenv, ]

\matrix [matrix of math nodes,nodes in empty cells,left delimiter={[},right delimiter={]},inner sep=1pt,outer sep=1.5pt,column sep=2pt,row sep=2pt,nodes={minimum width=15pt,anchor=center,inner sep=0pt,outer sep=0pt},inner sep=4pt,
column 3/.style={text width=4.1em, align=center},
column 6/.style={text width=4.1em, align=center},
column 7/.style={text width=4.1em, align=center},
column 10/.style={text width=4.1em, align=center},
row 1/.style={text centered,minimum height = 2em},
row 2/.style={text centered,minimum height = 2em},
row 3/.style={text centered,minimum height = 2em}, inner sep=4pt,label={[]right:$\;\;\;\,,$}] (m)
{
&&  
\mathbf{0}_{1\times{(m-1)}}&-\Phi&
I&\mathbf{0}_{1\times{(m-1)}}&
&&
&&
&  \\
&&  
&&
\ddots&&
\ddots& &
&&
&  \\
&& 
&&
&&
\mathbf{0}_{1\times{(m-1)}}&-\Phi&
I&\mathbf{0}_{1\times{(m-1)}}&
&  \\
};

\mymatrixbraceright{1}{3}{$k$}
\mymatrixbracetop{1}{2}{$\scriptscriptstyle (p-k)m$}
\mymatrixbracetop{3}{4}{$\scriptscriptstyle m$}
\mymatrixbracetop{5}{6}{$\scriptscriptstyle m$}
\mymatrixbracetop{11}{12}{$\scriptscriptstyle N_t-(p+1)m$}

\draw[black, thick, dashed] (m-1-2.north east) -- (m-3-2.south east); 
\draw[black, thick, dashed] (m-1-5.north west) -- (m-3-5.south west);    
\draw[black, thick, dashed] (m-1-6.north east) -- (m-3-6.south east);  
\draw[black, thick, dashed] (m-1-9.north west) -- (m-3-9.south west);
\draw[black, thick, dashed] (m-1-11.north west) -- (m-3-11.south west);

\draw[black, thick, dashed] (m-1-1.south west) -- (m-1-12.south east);  
\draw[black, thick, dashed] (m-2-1.south west) -- (m-2-12.south east);

\end{tikzpicture}
\end{equation}
\end{center}
which initially contains $(p-k)m+m$ columns corresponding to the omitted points on the local coarse grid. The following $km$ columns correspond to the $C$-points present on the local coarse grid and their corresponding following interval of $F$-points. Next, we consider
\begin{center}
\begin{tikzpicture}[baseline={-0.5ex},mymatrixenv]

\matrix [mymatrix,inner sep=4pt,label={[]left:$P_S^{(p)}(A_c^{(p)})^{-1}=$},label={[]right:$\; \; \; \,$,}] (m)
{
&                &&&  \\

\Phi^{(k-1)m}&\hdots& \Phi^{2m}               &\Phi^m&I  \\
\Phi^{(k-1)m+1}&\hdots& \Phi^{2m+1}               &\Phi^{m+1}&\Phi  \\
\vdots&& \vdots              &\vdots&\vdots  \\
\Phi^{km-1}&\hdots& \Phi^{3m-1}               &\Phi^{2m-1}&\Phi^{m-1}  \\

&      && &\\
};

\mymatrixbraceright{1}{1}{$\scriptscriptstyle pm$}
\mymatrixbraceright{2}{5}{$\scriptscriptstyle m$}
\mymatrixbraceright{6}{6}{$\scriptscriptstyle N_t-(p+1)m$}
\mymatrixbracetop{1}{5}{$\scriptscriptstyle k$}

\draw[black, thick, dashed] (m-1-1.south west) -- (m-1-5.south east);  
\draw[black, thick, dashed] (m-5-1.south west) -- (m-5-5.south east);  
\end{tikzpicture}
\end{center}
with $A_c^{(p)}$ as in \eqref{eq:RAP},
which defines the effect of selective ideal interpolation multiplied by the inverse of the coarse-grid problem. Due to the selective ideal interpolation operator, exactly $m$ points are considered, namely the $C$-point to be updated and the following $F$-interval consisting of $m-1$ points. All other points are not changed by the update of one $p$ and the corresponding rows are therefore zero.
As a consequence, the product $P_S^{(p)}(A_c^{(p)})^{-1}R_I^{(p)}A$ also has only $m$ nonzero rows. Furthermore, we have exactly $k+1$ blocks of $m \times m$ matrices which are not equal to zero. The matrix $P_S^{(p)}(A_c^{(p)})^{-1}R_I^{(p)}A$ in block form is given by
\begin{center}
\begin{equation}\label{eq:a2}
\begin{tikzpicture}[
baseline={-0.5ex},
mymatrixenv]

\matrix [matrix of math nodes,nodes in empty cells,left delimiter={[},right delimiter={]},inner sep=1pt,outer sep=1.5pt,column sep=2pt,row sep=2pt,nodes={anchor=center,inner sep=0pt,outer sep=0pt},inner sep=4pt,label={[]right:$\; \; \; \, ,$},
column 1/.style={text width=3em, align=center},
column 2/.style={text width=4em, align=center},
column 3/.style={text width=4em, align=center},
column 4/.style={text width=4em, align=center},
column 5/.style={text width=4em, align=center},
column 6/.style={text width=4em, align=center},
column 7/.style={text width=3em, align=center},
row 1/.style={minimum height=2em, align=center},
row 2/.style={text centered,minimum height = 3em},
row 3/.style={minimum height=2em, align=center}] (m)
{
&&&&&&  \\
& \mathcal{D} &\mathcal{V}_{k-2} &\hdots &\mathcal{V}_0 & \mathcal{S} &  \\
&&&&&&  \\
};

\mymatrixbraceright{1}{1}{$\scriptscriptstyle pm$}
\mymatrixbraceright{2}{2}{$\scriptscriptstyle m$}
\mymatrixbraceright{3}{3}{$\scriptscriptstyle N_t-(p+1)m$}

\mymatrixbracetop{1}{1}{$\scriptscriptstyle (p-k)m$}
\mymatrixbracetop{2}{2}{$\scriptscriptstyle m$}
\mymatrixbracetop{3}{5}{$\scriptscriptstyle (k-1)m$}
\mymatrixbracetop{6}{6}{$\scriptscriptstyle m$}
\mymatrixbracetop{7}{7}{$\scriptscriptstyle N_t-(p+1)m$}

\draw[black, thick, dashed] (m-1-1.north east) -- (m-3-1.south east);     
\draw[black, thick, dashed] (m-1-2.north east) -- (m-3-2.south east);
\draw[black, thick, dashed] (m-1-3.north east) -- (m-3-3.south east);
\draw[black, thick, dashed] (m-1-4.north east) -- (m-3-4.south east);
\draw[black, thick, dashed] (m-1-5.north east) -- (m-3-5.south east);
\draw[black, thick, dashed] (m-1-6.north east) -- (m-3-6.south east);

\draw[black, thick, dashed] (m-1-1.south west) -- (m-1-7.south east);  
\draw[black, thick, dashed] (m-2-1.south west) -- (m-2-7.south east);  

\end{tikzpicture}
\end{equation}
\end{center}
with blocks
\begin{center}
\begin{equation}\label{eq:exact_s_d_v}
\begin{tikzpicture}[baseline={-0.5ex},mymatrixenv]

\matrix [mymatrix,inner sep=4pt,label={[]left:$\mathcal{S}=$},label={[]right:,}] (m)
{
	I 	       & \mathbf{0} & \hdots & \mathbf{0}  \\
	\Phi       & \vdots & & \vdots  \\
	\vdots     & \vdots & & \vdots \\
	\Phi^{m-1} & \mathbf{0} & \hdots & \mathbf{0} \\
};     

\matrix [mymatrix,inner sep=4pt,label={[]left:$\mathcal{D}=$},label={[]right:,}] (m) at (6,0)
{
	\mathbf{0}&\hdots&\mathbf{0}& -\Phi^{(k-1)m+1}  \\
	\mathbf{0}&\hdots&\mathbf{0}& -\Phi^{(k-1)m+2}  \\
	\vdots    &      &\vdots    & \vdots  \\
	\mathbf{0}&\hdots&\mathbf{0}& -\Phi^{km} \\
};     

\matrix [mymatrix,inner sep=4pt,label={[]left:$\mathcal{V}_x=$},label={[]right:.}] (m) at (2,-3)
{
	\Phi^{(x+1)m}  &\mathbf{0}&\hdots&\mathbf{0}& -\Phi^{xm+1}  \\
	\Phi^{(x+1)m+1}&\mathbf{0}&\hdots&\mathbf{0}& -\Phi^{xm+2}  \\
	\vdots         &\vdots    &      &\vdots    & \vdots  \\
	\Phi^{(x+2)m-1}&\mathbf{0}&\hdots&\mathbf{0}& -\Phi^{(x+1)m} \\
};

\end{tikzpicture}
\end{equation}
\end{center}
Here, $\mathcal{D}$ comes from the truncated coarse-grid points, $\mathcal{V}_{k-2},\hdots, \mathcal{V}_0$ represent the first $k-1$ local coarse-grid points, and $\mathcal{S}$ corresponds to the last point of the local coarse grid.  Note, the $\mathcal{S}$-block is the diagonal block of the larger matrix. Now, we consider the operator $PR_I$, which is equivalent to an $F$-relaxation, and globally given by
\begin{equation*}
PR_I = \begin{bmatrix}
\mathcal{S}&&  \\

&\ddots& \\
&&\mathcal{S} \end{bmatrix}.
\end{equation*}
We can now calculate the error-propagation $(I - \sum_{p=0}^{N_T} P_S^{(p)}( A_c^{(p)})^{-1} R_I^{(p)} A)PR_I$ by exploiting the structure of the matrices $P_S^{(p)}(A_c^{(p)})^{-1}R_I^{(p)}A$ and $PR_I$. Instead of computing the complete matrix, we can compute the blocks $-\mathcal{D}\mathcal{S}, -\mathcal{V}_x\mathcal{S}$ for $x=k-2,\hdots,0$, and $(I-\mathcal{S})\mathcal{S}$. Note that the identity term is added to $-\mathcal{S}$ because $\mathcal{S}$ is the diagonal block of $P_S^{(p)}(A_c^{(p)})^{-1}R_I^{(p)}A$. Working through the algebra yields $-\mathcal{V}_x\mathcal{S}=\mathbf{0}$ for $x=k-2,\hdots,0$, $(I-\mathcal{S})\mathcal{S} = \mathcal{S}^2 - \mathcal{S} =\mathbf{0}$, and
\begin{center}
\begin{tikzpicture}[baseline={-0.5ex},mymatrixenv]

\matrix [mymatrix,inner sep=4pt,label={[]left:$-\mathcal{D}\mathcal{S}=$}] (m) at (6,0)
{
\Phi^{km} & \mathbf{0} & \hdots & \mathbf{0} \\
\Phi^{km+1} & \mathbf{0} & \hdots & \mathbf{0} \\
\vdots & \vdots & \vdots & \vdots \\
\Phi^{km+m-1} & \mathbf{0} & \hdots & \mathbf{0} \\
};  
\end{tikzpicture},
\end{center}
which is equivalent to the definition of $\mathcal{G}$ in \eqref{eq:error_exact}. Note that for the case $p<k$ in matrix \eqref{eq:a2} the operator $\mathcal{D}$ is omitted, since for these $C$-points all previous $C$-points are contained in the local coarse grid. 

\section{Error propagation for approximate local coarse problems}\label{app:AppendixB}

The definition $R_I^{(p)}A$ is the same as in \eqref{eq:a1}, but now
\begin{center}
\begin{tikzpicture}[baseline={-0.5ex},mymatrixenv]

\matrix [mymatrix,inner sep=4pt,label={[]left:$P_S^{(p)}\widetilde{A_c}^{(p)}=$}] (m)
{
&                &&&  \\

\Phi^{0}\Psi^{k-1}           &\hdots& \Phi^{0}\Psi^{2}            &\Phi^{0}\Psi           &\Phi^{0}  \\

\vdots               &      & \vdots              &\vdots         &\vdots  \\
\Phi^{m-1}\Psi^{k-1} &\hdots& \Phi^{m-1}\Psi^{2}  &\Phi^{m-1}\Psi &\Phi^{m-1}  \\

&                &&&  \\
};

\mymatrixbraceright{1}{1}{$\scriptscriptstyle pm$}
\mymatrixbraceright{2}{4}{$\scriptscriptstyle m$}
\mymatrixbraceright{5}{5}{$\scriptscriptstyle N_t-(p+1)m$}
\mymatrixbracetop{1}{5}{$\scriptscriptstyle k$}

\draw[black, thick, dashed] (m-1-1.south west) -- (m-1-5.south east);  
\draw[black, thick, dashed] (m-4-1.south west) -- (m-4-5.south east);  

\end{tikzpicture}.
\end{center}
As a result, we get a block matrix equivalent to \eqref{eq:a2}, but this time with block matrices $\widetilde{\mathcal{V}_x}$ and $\widetilde{\mathcal{D}}$ given by
\begin{center}
\begin{tikzpicture}[baseline={-0.5ex},mymatrixenv]

\matrix [mymatrix,inner sep=4pt,label={[]left:$\widetilde{\mathcal{V}_x}=$}] (m)
{
	\Phi^{0}\Psi^{(x+1)}  &\mathbf{0}&\hdots&\mathbf{0}& -\Phi^{0}\Psi^{x} \Phi  \\
	
	\vdots         &\vdots    &      &\vdots    & \vdots  \\
	\Phi^{m-1}\Psi^{(x+1)}&\mathbf{0}&\hdots&\mathbf{0}& -\Phi^{m-1} \Psi^{x} \Phi \\
};     

\matrix [mymatrix,inner sep=4pt,label={[]left:$\widetilde{\mathcal{D}}=$}] (m) at (6.3,0)
{
	\mathbf{0}&\hdots&\mathbf{0}& -\Phi^{0}\Psi^{k-1} \Phi  \\
	
	\vdots    &      &\vdots    & \vdots  \\
	\mathbf{0}&\hdots&\mathbf{0}& -\Phi^{m-1} \Psi^{k-1} \Phi\\
};     

\end{tikzpicture}
\end{center}
Note, that $\mathcal{S}$ is the same as given in \eqref{eq:exact_s_d_v}. Again, we use the structure of the matrices and calculate block submatrices of $P_S^{(p)}(\widetilde{A_c}^{(p)})^{-1}R_I^{(p)}APR_I$ given by
\begin{center}
\begin{tikzpicture}[baseline={-0.5ex},mymatrixenv]

\matrix [mymatrix,inner sep=4pt,label={[]left:$-\widetilde{\mathcal{V}_x}\mathcal{S} =$}] (m)
{
	\Phi^{0}\Psi^{x}(\Phi^{m}-\Psi) & \mathbf{0} & \hdots & \mathbf{0} \\
	
	\vdots & \vdots & \vdots & \vdots\\
	\Phi^{m-1}\Psi^{x}(\Phi^{m}-\Psi) & \mathbf{0} & \hdots & \mathbf{0} \\
};

\matrix [mymatrix,inner sep=4pt,label={[]left:$-\widetilde{\mathcal{D}}\mathcal{S} =$}] (m)at  (6.5,0)  
{
	\Phi^{0}\Psi^{k-1}\Phi^{m} & \mathbf{0} & \hdots & \mathbf{0} \\
	
	\vdots & \vdots & \vdots & \vdots\\
	\Phi^{m-1}\Psi^{k-1}\Phi^{m} & \mathbf{0} & \hdots & \mathbf{0} \\
};

\end{tikzpicture}
\end{center}
where $-\widetilde{\mathcal{V}_x}\mathcal{S}$ is equivalent to $\mathcal{Z}_x$ from \eqref{eq:error_inexact_z_w} and $-\widetilde{\mathcal{D}}\mathcal{S}$ is equivalent to $\mathcal{W}$ from \eqref{eq:error_inexact_z_w}.

\section{Discussion spatial parallelism}\label{app:AppendixC}

Here we demonstrate that the use of spatial parallelism has comparable effects on sequential time-stepping (before saturation) as it does on \namealg. In particular, we emphasize that when spatial parallelism saturates, the observed near-perfect speedup obtained by spatial parallelism before saturation will extend to \namealgnospace. \Cref{tab:spatial_parallelism} presents overall runtimes for using one and four processes in space for time-stepping, Parareal, and two-level \namealgnospace, the last two using 64 processes in time (same variants as in \Cref{tab:gray_two_level}). We see that for all algorithms we get a speedup of about 3.8 by using four spatial processes compared to one process. Note that this problem scales well with spatial parallelization, and spatial parallelism (as in most cases) should be the first choice. However, spatial parallelization is exhausted at some point and temporal parallelization can then provide additional speedups. 

\begin{table*}
	\renewcommand{\arraystretch}{1.3}
	\begin{center}
	\begin{tabular}{ |c|c|c|c|c|c|c|c|c|}
		\hline
		\multirow{2}{*}{\parbox{2.2cm}{\centering Space \\ processes}} & \multirow{2}{*}{\parbox{2.7cm}{\centering Time-stepping \\ one time process}}  & \multirow{2}{*}{\parbox{2.7cm}{ \centering Parareal \\ 64 time processes}} & \multirow{2}{*}{\parbox{3.5cm}{\centering Two-level \namealg \\ 64 time processes}} \\
		 	   &            &     &  \\ \hline
1 &    347{,}666 s   &  91{,}596 s & 79{,}397 s     \\ \hline
4 &    92{,}473 s        & 23{,}708 s &    20{,}411 s   \\ \hline
		\end{tabular}
	\end{center}
	\caption{Total runtimes using one and four processes in space for time-stepping, Parareal, and \namealg with $k=32$, the latter two using a coarsening factor of 256 and 64 processes in time.}
	\label{tab:spatial_parallelism}
	\vspace{-2ex}
\end{table*}

\section*{Acknowledgments}
We would like to thank Dr. Wayne Mitchell for helpful discussions that initiated this research. Los Alamos National Laboratory report number LA-UR-21-26105.

\bibliographystyle{siamplain}
\bibliography{references}
\end{document}